\documentclass[12pt,leqno]{amsart}
\usepackage[hidelinks]{hyperref}
\usepackage{pifont}
\usepackage{stmaryrd}
\usepackage{dsfont}
\usepackage{color}
\usepackage{mathrsfs}
\usepackage{amsmath}
\usepackage{amssymb}
\usepackage{bookmark}
\usepackage[bottom]{footmisc}
\usepackage{verbatim}
\usepackage{extarrows}
\usepackage{mathtools}
\usepackage{orcidlink}

\numberwithin{equation}{section}
\newtheorem{thm}{Theorem}[section]
\newtheorem{lem}[thm]{Lemma}

\newtheorem{prop}[thm]{Proposition}             
\newtheorem{cor}[thm]{Corollary}
\newtheorem{defi}[thm]{Definition}
\newtheorem{rmk}[thm]{Remark}

\newtheorem{conj}[thm]{Conjecture}

\newcommand{\End}{\operatorname{End}}

\newcommand{\xqz}[1]{\lfloor #1 \rfloor}
\newcommand{\Res}{\operatorname{Res}}

\renewcommand{\mod}{\operatorname{mod}}

\def\Z{\mathbb{Z}}
\def\N{\mathbb{N}}

\def\C{\mathbb{C}}

\def\D{\mathcal{D}}
\def\bm{\bar{m}}
\def\bp{\bar{p}}
\def\bn{\bar{n}}
\def\xm{\xqz{m}}
\def\xp{\xqz{p}}
\def\xn{\xqz{n}}

\allowdisplaybreaks

\begin{document}

\title[Twisted $\phi$-coordinated modules  and Zhu's  correspondence theorem]{Twisted $\phi$-coordinated modules for vertex algebras and Zhu's  correspondence theorem}

\author{Shun Xu~\orcidlink{0009-0006-8080-8107}}

\address{School of Mathematical Sciences, Tongji University, Shanghai, 200092, China}

\email{shunxu@tongji.edu.cn}



\begin{abstract}
Let $V$ be a vertex algebra and  $g$ be an automorphism of $V$ of order $T$. For any $n, m \in (1/T)\mathbb{N}$, we construct an $\tilde{A}_{g,n}(V)\!-\!\tilde{A}_{g,m}(V)$-bimodule  $\tilde{A}_{g,n,m}(V)$, where $\tilde{A}_{g,n}(V)$ denotes the associative algebra constructed by the authors in \cite{Shun1}. We introduce the notion of $(1/T)\mathbb{N}$-graded $g$-twisted $\phi$-coordinated $V$-modules and  prove that there exists a bijection between the simple $\tilde{A}_{g}(V)$-modules and the irreducible $(1/T)\mathbb{N}$-graded $g$-twisted $\phi$-coordinated $V$-modules, where $\tilde{A}_{g}(V)=\tilde{A}_{g,0}(V)$.  We construct the universal enveloping algebra $U(V[g])$, showing that  $\tilde{A}_{g}(V)$ is
subquotient of $U(V[g])$.
When $V$ is vertex operator algebra, we show that each $\tilde{A}_{g,n,m}(V)$ is isomorphic to the $A_{g,n}(V)-A_{g,m}(V)$-bimodule $A_{g,n,m}(V)$ constructed by Dong and Jiang~\cite{DJ2}. Also we prove that  there exists a bijection between the irreducible admissible $g$-twisted $V$-modules and the irreducible $(1/T)\mathbb{N}$-graded $g$-twisted $\phi$-coordinated $V$-modules.
\end{abstract}

\maketitle
\tableofcontents

\section{Introduction}

 Let $V$ be a vertex operator algebra. Zhu~\cite{Z1} constructed an associative algebra $A(V)$, now known as the \emph{Zhu algebra}, and proved a bijection between irreducible $A(V)$-modules and irreducible admissible $V$-modules. Dong, Li, and Mason extended this framework in a series of influential works. In~\cite{DLM1}, for each $n \in \mathbb{N}$, they introduced an associative algebra $A_n(V)$ satisfying $A_0(V) = A(V)$. In~\cite{DLM2}, they generalized Zhu’s construction to twisted representations: given a finite-order automorphism $g$ of $V$, they defined an associative algebra $A_g(V)$, which coincides with $A(V)$ when $g = \operatorname{id}_V$. Finally, in~\cite{DLM3}, they unified these results by constructing, for an automorphism $g$ of order $T$ and for $n \in (1/T)\mathbb{N}$, an associative algebra $A_{g,n}(V)$ such that $A_{g,0}(V)=A_{g}(V)$ and $A_{\operatorname{id}_V,n}(V)=A_{n}(V)$, and established a bijection between irreducible $A_{g,n}(V)$-modules that do not factor through $A_{g,n-1/T}(V)$ and irreducible admissible $g$-twisted $V$-modules.

In the same work~\cite{DLM3}, for any $A_{g,m}(V)$-module $U$ that does not factor through $A_{g,m-1/T}(V)$, they constructed a Verma-type $g$-twisted admissible $V$-module $\bar{M}_m(U)$. Although $\bar{M}_m(U)$ is naturally $(1/T)\mathbb{N}$-graded, the explicit structure of its homogeneous components remained unclear. To resolve this, Dong and Jiang developed a bimodule theory in~\cite{DJ2}. For any $n, m \in (1/T)\mathbb{N}$, they constructed an $A_{g,n}(V)$--$A_{g,m}(V)$-bimodule $A_{g,n,m}(V)$, satisfying $A_{g,n,n}(V) = A_{g,n}(V)$, thereby extending Zhu’s algebra. Given an $A_{g,m}(V)$-module $U$, they showed that the induced module
$
\bigoplus_{n \in (1/T)\mathbb{N}} A_{g,n,m}(V) \otimes_{A_{g,m}(V)} U
$
is a Verma-type $g$-twisted admissible $V$-module isomorphic to $\bar{M}_m(U)$.

In Lie algebra representation theory, the universal enveloping algebra serves as a fundamental structural and computational tool. Analogously, for a vertex operator algebra $V$ equipped with a finite-order automorphism $g$, there exists a weak analogue of the universal enveloping algebra, denoted $U(V[g])$, such that every weak $g$-twisted $V$-module naturally becomes a $U(V[g])$-module. The algebras $A_{g,n}(V)$ and bimodules $A_{g,n,m}(V)$ are deeply connected to this universal enveloping algebra. Frenkel and Zhu~\cite{FZ1} observed that Zhu’s algebra $A(V)$ can be realized as a certain subquotient of $U(V[1])$. This perspective was later extended: in the untwisted case ($g = \mathrm{id}_V$), it was shown in~\cite{HX1} that each $A_{g,n}(V)$ with $n \in \mathbb{N}$ arises as a subquotient of $U(V[1])$, and~\cite{HX2} generalized this result to arbitrary finite-order automorphisms $g$. Furthermore, the bimodule $A_{1,n,m}(V)$ was also identified as a subquotient of $U(V[1])$ (see~\cite{HJZ1}). Most recently, in collaboration with Han and Xiao~\cite{HXX1}, the author proved that for a general finite-order automorphism $g$, the $A_{g,n}(V)$--$A_{g,m}(V)$-bimodule $A_{g,n,m}(V)$ is likewise a subquotient of the universal enveloping algebra $U(V[g])$.

For parallel developments in the context of vertex operator superalgebras, we refer the reader to~\cite{KW1,DZ1,Shun2,Shun3}.

Li~\cite{L2} constructed a family of associative algebras $\tilde{A}_n(V)$ for a weak quantum vertex algebra $V$ equipped with a constant $\mathcal{S}$-map, where ordinary vertex algebras satisfy these conditions. He introduced the notion of $\mathbb{N}$-graded $\phi$-coordinated $V$-modules and, under suitable assumptions, established a bijection between equivalence classes of irreducible $\mathbb{N}$-graded $\phi$-coordinated $V$-modules and isomorphism classes of irreducible $\tilde{A}_0(V)$-modules.

We now turn to the more general setting of vertex algebras. Let $V$ be a vertex algebra and $g$ an automorphism of $V$ of order $T$. For any $n \in (1/T)\mathbb{N}$, the author constructed in~\cite{Shun1} an associative algebra $\tilde{A}_{g,n}(V)$, which reduces to Li’s $\tilde{A}_n(V)$ when $g = \operatorname{id}_V$. When $V$ is a vertex operator algebra, $\tilde{A}_{g,n}(V)$ is isomorphic to the algebra $A_{g,n}(V)$ introduced by Dong, Li, and Mason in~\cite{DLM3}. Combining this with results from~\cite{DLM3}, the author further proved that the notions of $g$-rationality and $g$-regularity for vertex operator algebras are independent of the choice of conformal vector.

Motivated by Li’s work~\cite{L2}, we introduce the notion of $(1/T)\mathbb{N}$-graded $g$-twisted $\phi$-coordinated $V$-modules and establish a bijection between simple $\tilde{A}_{g }(V)$-modules and irreducible $(1/T)\mathbb{N}$-graded $g$-twisted $\phi$-coordinated $V$-modules, where $\tilde{A}_{g }(V)=\tilde{A}_{g,0 }(V)$. In this process, for any $\tilde{A}_{g }(V)$-module $U$, we construct a Verma-type $(1/T)\mathbb{N}$-graded $g$-twisted $\phi$-coordinated $V$-module $\bar{M}(U)$.

We construct the universal enveloping algebra $U(V[g])$ using the $g$-twisted $\phi$-Jacobi identity, ensuring that every $g$-twisted $\phi$-coordinated $V$-module naturally carries the structure of a $U(V[g])$-module. Moreover, we prove that   the associative algebra $\tilde{A}_{g }(V)$ arises as subquotients of $U(V[g])$.

For any $n, m \in (1/T)\mathbb{N}$, we construct an $\tilde{A}_{g,n}(V)$--$\tilde{A}_{g,m}(V)$-bimodule $\tilde{A}_{g,n,m}(V)$. When $V$ is a vertex operator algebra, we show that each $\tilde{A}_{g,n,m}(V)$ is isomorphic to the $A_{g,n}(V)$--$A_{g,m}(V)$-bimodule $A_{g,n,m}(V)$ constructed by Dong and Jiang~\cite{DJ2}. By combining the Zhu-type correspondence theorem from~\cite{DLM3,DLM2} with the one established in this paper, we further prove a bijection between irreducible admissible $g$-twisted $V$-modules and irreducible $(1/T)\mathbb{N}$-graded $g$-twisted $\phi$-coordinated $V$-modules.

The organization of this paper is as follows.  
In Section~2, we review the basic theory of vertex algebras and $g$-twisted $\phi$-coordinated $V$-modules, establish several preliminary results needed in subsequent sections, and introduce the notion of $(1/T)\mathbb{N}$-graded $g$-twisted $\phi$-coordinated $V$-modules.  
In Section~3, for any $n, m \in (1/T)\mathbb{N}$, we construct the $\tilde{A}_{g,n}(V)$--$\tilde{A}_{g,m}(V)$-bimodule $\tilde{A}_{g,n,m}(V)$.  
In Section~4, we define a functor $\tilde{\Omega}_n$ from the category of $g$-twisted $\phi$-coordinated $V$-modules to the category of $\tilde{A}_{g,n}(V)$-modules.  
In Section~5, we construct the Lie algebra $\hat{V}[g]$, which plays a crucial role in the construction of the functor $\tilde{L} $.  
In Section~6, we construct the functor $\tilde{L} $ from the category of $\tilde{A}_{g }(V)$-modules to the category of $g$-twisted $\phi$-coordinated $V$-modules.   
In Section~7, we construct the universal enveloping algebra $U(V[g])$, showing that every $g$-twisted $\phi$-coordinated $V$-module naturally carries the structure of a $U(V[g])$-module, and establish a precise relationship between   $\tilde{A}_{g }(V)$ and  $U(V[g])$.  
In Section~8, assuming that $V$ is a vertex operator algebra, we prove that each $\tilde{A}_{g,n,m}(V)$ is isomorphic to the bimodule $A_{g,n,m}(V)$ constructed by Dong and Jiang~\cite{DJ2}, and further establish a bijection between irreducible admissible $g$-twisted $V$-modules and irreducible $(1/T)\mathbb{N}$-graded $g$-twisted $\phi$-coordinated $V$-modules.\section{Preliminaries}

In this section, we review the basic theory of vertex algebras and $g$-twisted $\phi$-coordinated $V$-modules, establish several preliminary results required in subsequent sections, and introduce the notion of $(1/T)\mathbb{N}$-graded $g$-twisted $\phi$-coordinated $V$-modules.

\begin{defi}
A \emph{vertex algebra} is a vector space $V$ equipped with a linear map
\begin{align*}
Y(\cdot, x) \colon V &\to (\operatorname{End} V)[[x, x^{-1}]] \\
v &\mapsto Y(v, x) = \sum_{n \in \mathbb{Z}} v_n x^{-n-1},
\end{align*}
and a distinguished vector $\mathbf{1} \in V$, called the \emph{vacuum vector}, satisfying the following axioms for all $u, v \in V$:

{\rm(1)} $Y(\mathbf{1}, z) = \operatorname{id}_V$, and $u_n \mathbf{1} = \delta_{n,-1} u$ for all $n \geq -1$;
    
{\rm(2)} $u_n v = 0$ for all sufficiently large $n$;
    
{\rm(3)} the \emph{Jacobi identity}:
    $$
    \begin{gathered}
    z_0^{-1} \delta\!\left(\frac{z_1 - z_2}{z_0}\right) Y(u, z_1) Y(v, z_2)
    - z_0^{-1} \delta\!\left(\frac{z_2 - z_1}{-z_0}\right) Y(v, z_2) Y(u, z_1) \\
    = z_2^{-1} \delta\!\left(\frac{z_1 - z_0}{z_2}\right) Y(Y(u, z_0)v, z_2).
    \end{gathered}
    $$
\end{defi}

Let $V$ be a vertex algebra. Define a linear operator $\mathcal{D} \colon V \to V$ by
\[
\mathcal{D}(v) = v_{-2}\mathbf{1}
= \lim_{x \to 0} \frac{d}{dx} Y(v, x)\mathbf{1},
\qquad \text{for all } v \in V.
\]
It follows from \cite{LL1} that for any $u, v \in V$, the following identities hold:
\begin{align}
[\mathcal{D}, Y(v, x)] &= Y(\mathcal{D}v, x) = \frac{d}{dx} Y(v, x), \label{eq:translation} \\
Y(u, x)v &= e^{x\mathcal{D}} Y(v, -x)u, \label{skew_sym} \\
(x_1 - x_2)^k Y(u, x_1) Y(v, x_2) &= (x_1 - x_2)^k Y(v, x_2) Y(u, x_1), \label{weak_asso}
\end{align}
where $k$ is a nonnegative integer depending on $u$ and $v$.

\begin{defi}
Let $(V, Y, \mathbf{1})$ be a vertex algebra. A linear isomorphism $g$ of $V$ is called an \emph{automorphism} of $V$ if it satisfies
\[
g(\mathbf{1}) = \mathbf{1} \quad \text{and} \quad g(Y(u, z)v) = Y(g(u), z)g(v) \quad \text{for all } u, v \in V.
\]
\end{defi}

Let $V$ be a vertex algebra, and fix $g$ to be an automorphism of $V$ of finite order $T$. Denote the imaginary unit by $\sqrt{-1}$. Then $V$ admits the eigenspace decomposition with respect to the action of $g$:
\[
V = \bigoplus_{r=0}^{T-1} V^r, \quad \text{where } V^r = \left\{ v \in V \mid gv = e^{-2\pi\sqrt{-1} r / T} v \right\}.
\]
In what follows, whenever we write $V^r$, we always assume $r \in \{0, 1, \ldots, T-1\}$. Note that $g$ commutes with $\mathcal{D}$, i.e., $g\mathcal{D} = \mathcal{D}g$, which implies $\mathcal{D}(V^r) \subset V^r$ for all $r$.

\begin{defi}\label{defi_coor_module}
A \emph{$g$-twisted $\phi$-coordinated $V$-module} is a vector space $W$ equipped with a linear map
\[
Y_W(\cdot, x): V \rightarrow \operatorname{Hom}\!\left(W, W\!\left(\left(x^{1/T}\right)\right)\right) \subset (\operatorname{End} W)\!\left[\left[x^{1/T}, x^{-1/T}\right]\right]
\]
satisfying the following conditions:

{\rm(1)} $Y_W(\mathbf{1}, x)=\operatorname{id}_W$ (the identity operator on $W$);
    
{\rm(2)}
    $
    Y_W(g v, x)=\lim _{x^{1/T} \rightarrow \omega_T^{-1} x^{1/T}} Y_W(v, x) 
    $ for any $v \in V$,
    where $\omega_T=e^{-2\pi\sqrt{-1}/T}$;
    
{\rm(3)}
    for any $u, v \in V$, there exists a nonnegative integer $k$ such that
    \begin{align*}
        \left(x_1-x_2\right)^k Y_W\left(u, x_1\right) Y_W\left(v, x_2\right) &\in \operatorname{Hom}\!\left(W, W\!\left(\left(x_1^{1/T}, x_2^{1/T}\right)\right)\right),\\
        x_2^k\left(e^{x_0}-1\right)^k Y_W\left(Y(u, {x_0}) v, x_2\right)&=\left.\left(\left(x_1-x_2\right)^k Y_W\left(u, x_1\right) Y_W\left(v, x_2\right)\right)\right|_{x_1^{1/T}=\left(x_2 e^{x_0}\right)^{1/T}}.
    \end{align*}
\end{defi}

\begin{rmk}
In the above definition, Axiom (2) is equivalent to the condition that $Y_W(u, x)w \in x^{-r/T} W((x))$ for any $u \in V^r$ and $w \in W$.
\end{rmk}

From \cite[Theorem 2.13]{LTW1} and \eqref{weak_asso}, we have:

\begin{thm}\label{jacobi}
Let $\left(W, Y_W\right)$ be a $g$-twisted $\phi$-coordinated $V$-module. Then the \emph{$g$-twisted $\phi$-Jacobi identity} holds:
\begin{align}
 \left(x_2 z\right)^{-1} \delta\!\left(\frac{x_1-x_2}{x_2 z}\right) &Y_W\left(u, x_1\right) Y_W\left(v, x_2\right) -\left(x_2 z\right)^{-1} \delta\!\left(\frac{x_2-x_1}{-x_2 z}\right)  Y_W\left(v, x_2\right) Y_W\left(u, x_1\right)\nonumber \\
 =&(1/T) \sum_{j=0}^{T-1} x_1^{-1} \delta\!\left(\omega_T^{-j}\left(\frac{x_2(1+z)}{x_1}\right)^{1/T}\right) Y_W\left(Y\left(g^j u, \log (1+z)\right) v, x_2\right), \nonumber
\end{align}
for any $u,v\in V$.
\end{thm}

In particular, when $u \in V^r$, the $g$-twisted $\phi$-Jacobi identity takes the following simplified form:
\begin{align}
    (x_2 z)^{-1} \delta\!\left(\frac{x_1 - x_2}{x_2 z}\right) &Y_W(u, x_1) Y_W(v, x_2) 
    - (x_2 z)^{-1} \delta\!\left(\frac{x_2 - x_1}{-x_2 z}\right) Y_W(v, x_2) Y_W(u, x_1) \nonumber \\
    &= x_1^{-1} \delta\!\left(\frac{x_2(1+z)}{x_1}\right) \left(\frac{x_2(1+z)}{x_1}\right)^{r/T} Y_W\left(Y(u, \log(1+z)) v, x_2\right). \label{phi_jacobi}
\end{align}
Applying the residue operator $\Res_z \, x_2$ to both sides of equation~\eqref{phi_jacobi} yields the commutator formula:
\begin{align}
    [Y_W(u,x_1),Y_W(v,x_2)]\label{comm_formula_zhengti}=&\Res_{x_0}x_2x_1^{-1} \delta\!\left(\frac{x_2(1+z)}{x_1}\right) \left(\frac{x_2(1+z)}{x_1}\right)^{r/T} \\&\times Y_W\left(Y(u, \log(1+z)) v, x_2\right).\nonumber
\end{align}
For any $m \in r/T + \mathbb{Z}$ and $n \in (1/T)\mathbb{Z}$, applying the residue operators $ \Res_{x_1} \Res_{x_2}\, x_1^m x_2^{n}$ to both sides of equation~\eqref{comm_formula_zhengti} gives:
\begin{equation}\label{comm_formula}
    [u_m, v_n] = \sum_{j \geq 0} \frac{(m+1)^j}{j!} (u_j v)_{m+n+1}.
\end{equation}

\begin{lem}\label{phi_weak_asso}
In the definition of a $g$-twisted $\phi$-coordinated $V$-module, the third axiom is equivalent to the following \emph{weak $g$-twisted $\phi$-associativity}:  
for any  $u \in V^r$,  $v \in V$ and  $w \in W$, there exists an integer $l \in \mathbb{N}$ such that
\[
(z+1)^{\,l + r/T} Y_W\!\big(u, (z+1)x_2\big) Y_W(v, x_2) w 
= (1+z)^{\,l + r/T} Y_W\!\big(Y(u, \log(1+z)) v, x_2\big) w.
\]
\end{lem}

\begin{proof}
Suppose $(W, Y_W)$ is a $g$-twisted $\phi$-coordinated $V$-module. Let $u \in V^r$, $v \in V$, and $w \in W$. Choose $l \in \mathbb{N}$ such that
\[
x_1^{\,l + r/T} Y_W(u, x_1) w \in W[[x_1]].
\]
Applying $\operatorname{Res}_{x_1} x_1^{\,l + r/T}$ to the $g$-twisted $\phi$-Jacobi identity \eqref{phi_jacobi} acting on $w$, we obtain
\[
(z+1)^{\,l + r/T} x_2^{\,l + r/T} Y_W\!\big(u, (z+1)x_2\big) Y_W(v, x_2) w
= (1+z)^{\,l + r/T} x_2^{\,l + r/T} Y_W\!\big(Y(u, \log(1+z)) v, x_2\big) w.
\]
Cancelling the common factor $x_2^{\,l + r/T}$ yields the weak $g$-twisted $\phi$-associativity.

Conversely, assume weak $g$-twisted $\phi$-associativity holds. Fix $u \in V^r$ and $v \in V$. Since $V$ is a vertex algebra, there exists $k \in \mathbb{N}$ such that
$
x^k Y(u, x) v \in V[[x]].
$
For any $w \in W$, choose $l \in \mathbb{N}$ so that weak associativity holds with this $l$. Multiplying both sides by $(z x_2)^k$, we get
\begin{align*}
& (z+1)^{\,l + r/T} (z x_2)^k Y_W\!\big(u, (z+1)x_2\big) Y_W(v, x_2) w \\
=\, & (1+z)^{\,l + r/T} (z x_2)^k Y_W\!\big(Y(u, \log(1+z)) v, x_2\big) w.
\end{align*}
The right-hand side involves only nonnegative powers of $z$, and hence so does the left-hand side. Then for any $j\in\N$, we have
\begin{align*}
   &\Res_{x_1}x_1^{l+r/T+j}(x_1-x_2)^kY_W(u,x_1)Y(v,x_2)w\\
   =&\Res_{x_1}\Res_z x_2(x_2z)^{-1}\delta\left(
   \frac{x_1-x_2}{zx_2}\right)x_1^{l+r/T+j}(x_1-x_2)^kY_W(u,x_1)Y(v,x_2)w\\
   =&\Res_{x_1}\Res_z x_2x_1^{-1}\delta\left(
   \frac{(z+1)x_2}{x_1}\right)((z+1)x_2)^{l+r/T+j}(zx_2)^kY_W(u,(z+1)x_2)Y(v,x_2)w\\
   =& \Res_z x_2 ((z+1)x_2)^{l+r/T+j}(zx_2)^kY_W(u,(z+1)x_2)Y(v,x_2)w=0.
\end{align*}
This implies
\begin{align}
   x_1^{\,l+r/T} (x_1 - x_2)^k Y_W(u, x_1) Y_W(v, x_2) w \in W((x_2^{1/T}))[[x_1]].\label{xiajieduan} 
\end{align}
Since $k$ is independent of $w$, it follows that
\[
(x_1 - x_2)^k Y_W(u, x_1) Y_W(v, x_2) \in \operatorname{Hom}\!\big(W, W((x_1^{1/T}, x_2^{1/T}))\big).
\]
Now consider the substitution $x_1^{1/T} = (x_2 e^{x_0})^{1/T}$, which can be realized as the composition
\[
x_1^{1/T} = ( x_2 + zx_2)^{1/T} \quad \text{with} \quad z = e^{x_0} - 1.
\]
Under this substitution, we compute:
\begin{align*}
& \left. (x_2 e^{x_0})^{\,l + r/T} \big( (x_1 - x_2)^k Y_W(u, x_1) Y_W(v, x_2) w \big) \right|_{x_1^{1/T} = (x_2 e^{x_0})^{1/T}} \\
=\, & \left. \left( \left. x_1^{\,l + r/T} (x_1 - x_2)^k Y_W(u, x_1) Y_W(v, x_2) w \right|_{x_1^{1/T} = ( x_2 +z x_2)^{1/T}} \right) \right|_{z = e^{x_0} - 1} \\
=\, & \left. \left( \left. x_1^{\,l + r/T} (x_1 - x_2)^k Y_W(u, x_1) Y_W(v, x_2) w \right|_{x_1^{1/T} = (z x_2 + x_2)^{1/T}} \right) \right|_{z = e^{x_0} - 1}\qquad\qquad \text{(by \eqref{xiajieduan})}\\
=\, & \left. \left(   (z x_2 + x_2)^{\,l + r/T} (z x_2)^k Y_W(u, z x_2 + x_2) Y_W(v, x_2) w   \right) \right|_{z = e^{x_0} - 1} \\
=\, & \left. ( x_2 + zx_2)^{\,l + r/T} (z x_2)^k Y_W\!\big(Y(u, \log(1+z)) v, x_2\big) w \right|_{z = e^{x_0} - 1} \\
=\, & (x_2 e^{x_0})^{\,l + r/T} (x_2 e^{x_0} - x_2)^k Y_W\!\big(Y(u, x_0) v, x_2\big) w.
\end{align*}
Hence $(W, Y_W)$ satisfies the axioms of a $g$-twisted $\phi$-coordinated $V$-module.
\end{proof}

\begin{rmk}
Combining the proof of Lemma~\ref{phi_weak_asso} with Theorem~\ref{jacobi}, we conclude that weak $g$-twisted $\phi$-associativity is equivalent to the $g$-twisted $\phi$-Jacobi identity.
\end{rmk}

We now discuss the notion of the submodule generated by a subset of a $g$-twisted $\phi$-coordinated $V$-module.  
Let $W$ be a $g$-twisted $\phi$-coordinated $V$-module and let $U \subseteq W$. Define $\langle U \rangle$ to be the smallest submodule of $W$ containing $U$; we call $\langle U \rangle$ the \emph{submodule generated by $U$}. Equivalently, $\langle U \rangle$ is the intersection of all submodules of $W$ that contain $U$, and it can be described explicitly as
\[
\langle U \rangle = \operatorname{span}\left\{ v^{(1)}_{n_1} \cdots v^{(r)}_{n_r} w \,\middle|\, r \in \mathbb{N},\; v^{(1)}, \dots, v^{(r)} \in V,\; n_1, \dots, n_r \in (1/T)\mathbb{Z},\; w \in U \right\}.
\]

\begin{prop}\label{prop2.8}
Let $W$ be a $g$-twisted $\phi$-coordinated $V$-module and let $U \subseteq W$. Then
\[
\langle U \rangle = \operatorname{span}\left\{ v_n w \,\middle|\, v \in V,\; n \in (1/T)\mathbb{Z},\; w \in U \right\}.
\]
\end{prop}

\begin{proof}
Let $u \in V^r$, $v \in V$, and $w \in U$. By weak $g$-twisted $\phi$-associativity (Lemma~\ref{phi_weak_asso}), there exists $l \in \mathbb{N}$ such that
\[
(z+1)^{\,l + r/T} Y_W\!\big(u, (z+1)x_2\big) Y_W(v, x_2) w 
= (1+z)^{\,l + r/T} Y_W\!\big(Y(u, \log(1+z)) v, x_2\big) w.
\]
The proposition follows from the following computation: for any $m \in r/T + \mathbb{Z}$ and $n \in (1/T)\mathbb{Z}$,
\begin{align*}
u_m v_n w
&= \Res_z\Res_{x_1} \Res_{x_2} x_1^m x_2^{n+1} (x_2 z)^{-1} \delta\!\left( \frac{x_1 - x_2}{x_2 z} \right) Y_W(u, x_1) Y_W(v, x_2) w \\
&= \Res_z\Res_{x_1} \Res_{x_2} ((z+1)x_2)^m x_2^{n+1} x_1^{-1} \delta\!\left( \frac{(z+1)x_2}{x_1} \right) Y_W(u, (z+1)x_2) Y_W(v, x_2) w \\
&= \Res_z\Res_{x_2} (z+1)^m x_2^{m+n+1} Y_W(u, (z+1)x_2) Y_W(v, x_2) w \\
&= \Res_z\Res_{x_2} (z+1)^{m - l - r/T} x_2^{m+n+1} \Big( (z+1)^{l + r/T} Y_W(u, (z+1)x_2) Y_W(v, x_2) w \Big) \\
&= \Res_z\Res_{x_2} (z+1)^{m - l - r/T} x_2^{m+n+1} \Big( (1+z)^{l + r/T} Y_W\!\big(Y(u, \log(1+z)) v, x_2\big) w \Big).
\end{align*}
Since the right-hand side lies in the span of vectors of the form $a_k w$ with $a \in V$, $k \in (1/T)\mathbb{Z}$, and $w \in U$, it follows inductively that all higher products reduce to linear combinations of single-mode actions. Hence $\langle U \rangle$ is already spanned by elements $v_n w$ with $v \in V$, $n \in (1/T)\mathbb{Z}$, and $w \in U$.
\end{proof}

From \cite[Lemma 2.10]{LTW1}, we have:

\begin{lem}\label{daozi}
Let  $(W, Y_W)$ be a $g$-twisted $\phi$-coordinated $V$-module. Then for any $v \in V$,
\[
Y_W(\mathcal{D} v, x) = \left( x \frac{d}{dx} \right) Y_W(v, x).
\]
\end{lem}

\begin{defi}
A \emph{$(1/T)\mathbb{N}$-graded $g$-twisted $\phi$-coordinated $V$-module} is a $g$-twisted $\phi$-coordinated $V$-module $W$ equipped with a grading
$
W = \bigoplus_{n \in (1/T)\mathbb{N}} W(n)
$
such that for all $v \in V$, $m \in (1/T)\mathbb{Z}$, and $n \in (1/T)\mathbb{N}$,
$
v_m \, W(n) \subseteq W(n - m - 1).
$
\end{defi}

\begin{rmk}
This definition differs slightly from Definition~2.6 in \cite{L2} due to Li. In what follows, whenever we refer to a nonzero $(1/T)\mathbb{N}$-graded $g$-twisted $\phi$-coordinated $V$-module $W$, we always assume that its lowest graded component is nontrivial, i.e., $W(0) \neq 0$.
\end{rmk}

\section{$\tilde{A}_{g,n}(V)\!-\!\tilde{A}_{g,m}(V)$-bimodule $\tilde{A}_{g,n,m}(V)$}

Let $V$ be a vertex algebra, and let $g$ be an automorphism of $V$ of finite order $T$.  
For $k, l \in \{0, 1, \ldots, T-1\}$, define
\[
\delta_k(l) =
\begin{cases}
1, & \text{if } k \geq l, \\
0, & \text{if } k < l,
\end{cases}
\qquad\text{and set }\delta_k(T) = 0.
\]
Every $n \in (1/T)\mathbb{Z}$ admits a unique decomposition
$
n = \lfloor n \rfloor + \bar{n}/T,
$
where $\bar{n} \in \{0,1,\dots,T-1\}$ and $\lfloor \cdot \rfloor$ denotes the floor function.
Let $u \in V^r$, $v \in V$, and $m,n,p \in (1/T)\mathbb{Z}$. The product $\bullet_{g,m,p}^{\,n}$ on $V$ is defined by
\begin{align*}
u \bullet_{g,m,p}^{\,n} v
&= \sum_{i=0}^{\lfloor p \rfloor} (-1)^i 
   \binom{\lfloor m \rfloor + \lfloor n \rfloor - \lfloor p \rfloor - 1 
         + \delta_{\bar{m}}(r) + \delta_{\bar{n}}(T - r) + i}{i} \\
&\quad \cdot \operatorname{Res}_{x} 
   \frac{e^{x(\lfloor m \rfloor + \delta_{\bar{m}}(r) + r/T)}}
        {(e^x - 1)^{\lfloor m \rfloor + \lfloor n \rfloor - \lfloor p \rfloor 
                   + \delta_{\bar{m}}(r) + \delta_{\bar{n}}(T - r) + i}} 
   Y(u,x)v \\
&= \sum_{i=0}^{\lfloor p \rfloor} (-1)^i 
   \binom{\lfloor m \rfloor + \lfloor n \rfloor - \lfloor p \rfloor - 1 
         + \delta_{\bar{m}}(r) + \delta_{\bar{n}}(T - r) + i}{i} \\
&\quad \cdot \operatorname{Res}_{y} 
   \frac{(1+y)^{-1 + \lfloor m \rfloor + \delta_{\bar{m}}(r) + r/T}}
        {y^{\lfloor m \rfloor + \lfloor n \rfloor - \lfloor p \rfloor 
            + \delta_{\bar{m}}(r) + \delta_{\bar{n}}(T - r) + i}} 
   Y\bigl(u,\log(1+y)\bigr)v,
\end{align*}
if $\bar{p} - \bar{n} \equiv r \pmod{T}$ and $m,n,p \geq 0$; otherwise we set $u \bullet_{g,m,p}^{\,n} v = 0$.
Set
$
\bar{\bullet}_{g,m}^{\,n} := \bullet_{g,m,n}^{\,n},$ 
$\bullet_{g,m}^{\,n} := \bullet_{g,m,m}^{\,n}$, 
$\bullet_{g,n} := \bullet_{g,n}^{\,n} = \bar{\bullet}_{g,n}^{\,n}
$ and $\bullet_g:=\bullet_{g,0}$.
When $r \neq 0$, one has $u \bar{\bullet}_{g,m}^{\,n} v = 0$. If $r = 0$, then
\[
u \bar{\bullet}_{g,m}^{\,n} v
= \sum_{i=0}^{\lfloor n \rfloor} (-1)^i \binom{\lfloor m \rfloor + i}{i}
  \operatorname{Res}_y \frac{(1+y)^{\lfloor m \rfloor}}{y^{\lfloor m \rfloor + i + 1}}
  Y\bigl(u,\log(1+y)\bigr)v.
\]
In particular, for any $v \in V$,
$
\mathbf{1} \bar{\bullet}_{g,m}^{\,n} v = v.
$
The product $\bullet_{g,m}^{\,n}$ vanishes unless $\bar{m} - \bar{n} \equiv r \pmod{T}$. In this case,
$
-1 + \delta_{\bar{m}}(r) + \delta_{\bar{n}}(T - r) = 0,
$
and consequently
\[
u \bullet_{g,m}^{\,n} v
= \sum_{i=0}^{\lfloor m \rfloor} (-1)^i \binom{\lfloor n \rfloor + i}{i}
  \operatorname{Res}_y 
  \frac{(1+y)^{-1 + \lfloor m \rfloor + \delta_{\bar{m}}(r) + r/T}}
       {y^{\lfloor n \rfloor + i + 1}}
  Y\bigl(u,\log(1+y)\bigr)v.
\]
For $m,n \in (1/T)\mathbb{N}$, define
\[
\tilde{O}_{g,n,m}'(V)
:= \operatorname{span}\{ u \diamond_{g,m}^{\,n} v \mid u,v \in V \} + L_{n,m}(V),
\]
where
\[
L_{n,m}(V) := \operatorname{span}\{ (\mathcal{D} + m - n)u \mid u \in V \},
\]
and for $u \in V^r$, $v \in V$,
\begin{align*}
u \diamond_{g,m}^{\,n} v
&= \operatorname{Res}_{x} 
   \frac{e^{x(\delta_{\bar{m}}(r) + \lfloor m \rfloor + r/T)}}
        {(e^x - 1)^{\lfloor m \rfloor + \lfloor n \rfloor 
                    + \delta_{\bar{m}}(r) + \delta_{\bar{n}}(T - r) + 1}} 
   Y(u,x)v \\
&= \operatorname{Res}_{y} 
   \frac{(1+y)^{-1 + \delta_{\bar{m}}(r) + \lfloor m \rfloor + r/T}}
        {y^{\lfloor m \rfloor + \lfloor n \rfloor 
            + \delta_{\bar{m}}(r) + \delta_{\bar{n}}(T - r) + 1}} 
   Y\bigl(u,\log(1+y)\bigr)v.
\end{align*}
Set $u \diamond_{g,n} v := u \diamond_{g,n}^{\,n} v$, $u\diamond_g v:=\diamond_{g,0}v$,  $\tilde{O}_{g,n}(V) := \tilde{O}_{g,n,n}'(V)$ and $\tilde{O}_{g}(V):=\tilde{O}_{g,0}(V)$.  
Finally, define the quotient space
\[
\tilde{A}_{g,n}(V) := V / \tilde{O}_{g,n}(V)\quad\text{ and }\quad\tilde{A}_{g }(V) := \tilde{A}_{g,0}(V).
\]
In our previous work~\cite{Shun1}, the following result was proved:

\begin{thm} \label{thm:associative_algebra_structure}
The product $\bullet_{g,n}$ induces the structure of an associative algebra on $\tilde{A}_{g,n}(V) = V / \tilde{O}_{g,n}(V)$, with identity element $\mathbf{1} + \tilde{O}_{g,n}(V)$.
\end{thm}

The proof of the following lemma is analogous to that of~\cite[Lemma~3.1]{Shun1}.

\begin{lem}\label{prop1}
Let $u \in V^r$, $v \in V$, and let $k, s$ be integers with $k \geq s \geq 0$. Then
\[
\operatorname{Res}_{x} \frac{e^{x(\delta_{\bar{m}}(r)+\lfloor m\rfloor+r / T+s)}}{(e^x-1)^{\lfloor m\rfloor+\lfloor n\rfloor+\delta_{\bar{m}}(r)+\delta_{\bar{n}}(T-r)+1+k}}\, Y(u, x)v \in \tilde{O}_{g,n,m}^\prime(V).
\]
\end{lem}

\begin{proof}
We begin by making the change of variables $y = e^x - 1$, then
\begin{align*}
&\quad\,\operatorname{Res}_x \frac{e^{x(\delta_{\bar{m}}(r)+\lfloor m\rfloor+r / T+s)}}{(e^x-1)^{\lfloor m\rfloor+\lfloor n\rfloor+\delta_{\bar{m}}(r)+\delta_{\bar{n}}(T-r)+1+k}} Y(u, x)v \\
&= \operatorname{Res}_y \frac{(1+y)^{\delta_{\bar{m}}(r)+\lfloor m\rfloor+r / T+s-1}}{y^{\lfloor m\rfloor+\lfloor n\rfloor+\delta_{\bar{m}}(r)+\delta_{\bar{n}}(T-r)+1+k}} Y(u, \log(1+y))v \\
&= \sum_{i=0}^s \binom{s}{i} \operatorname{Res}_y \frac{(1+y)^{\delta_{\bar{m}}(r)+\lfloor m\rfloor+r / T-1}}{y^{\lfloor m\rfloor+\lfloor n\rfloor+\delta_{\bar{m}}(r)+\delta_{\bar{n}}(T-r)+1+k-i}} Y(u, \log(1+y))v \\
&= \sum_{i=0}^s \binom{s}{i} \operatorname{Res}_x \frac{e^{x(\delta_{\bar{m}}(r)+\lfloor m\rfloor+r / T)}}{(e^x-1)^{\lfloor m\rfloor+\lfloor n\rfloor+\delta_{\bar{m}}(r)+\delta_{\bar{n}}(T-r)+1+k-i}} Y(u, x)v.
\end{align*}
Hence, it suffices to prove the statement for the case $s = 0$ and arbitrary $k \geq 0$. We proceed by induction on $k$.
The base case $k = 0$ follows directly from the definition of $\tilde{O}_{g,n,m}^\prime(V)$.
Assume the claim holds for all integers $k \leq l$. Consider $k = l + 1$. Using the inductive hypothesis and the  identity \eqref{eq:translation}, we compute:
\begin{align*}
& \operatorname{Res}_x \frac{e^{x(\delta_{\bar{m}}(r) + \lfloor m \rfloor + r/T)}}{(e^x - 1)^{\lfloor m\rfloor+\lfloor n \rfloor + \delta_{\bar{m}}(r) + \delta_{\bar{n}}(T - r) + 1 + l}} Y(\mathcal{D}u, x)v \\
= & \operatorname{Res}_x \frac{e^{x(\delta_{\bar{m}}(r) + \lfloor m \rfloor + r/T)}}{(e^x - 1)^{\lfloor m\rfloor+\lfloor n \rfloor + \delta_{\bar{m}}(r) + \delta_{\bar{n}}(T - r) + 1 + l}} \frac{d}{dx} Y(u, x)v \\
= & -\operatorname{Res}_x \left( \frac{d}{dx} \frac{e^{x(\delta_{\bar{m}}(r) + \lfloor m \rfloor + r/T)}}{(e^x - 1)^{\lfloor m\rfloor+\lfloor n \rfloor + \delta_{\bar{m}}(r) + \delta_{\bar{n}}(T - r) + 1 + l}} \right) Y(u, x)v \\
= & -(\delta_{\bar{m}}(r) + \lfloor m \rfloor + r/T) \operatorname{Res}_x \frac{e^{x(\delta_{\bar{m}}(r) + \lfloor m \rfloor + r/T)}}{(e^x - 1)^{\lfloor m\rfloor+\lfloor n \rfloor + \delta_{\bar{m}}(r) + \delta_{\bar{n}}(T - r) + 1 + l}} Y(u, x)v \\
& + (\lfloor m\rfloor+\lfloor n \rfloor + \delta_{\bar{m}}(r) + \delta_{\bar{n}}(T - r) + 1 + l) \operatorname{Res}_x \frac{e^{x(\delta_{\bar{m}}(r) + \lfloor m \rfloor + r/T+1)}}{(e^x - 1)^{\lfloor m\rfloor+\lfloor n \rfloor + \delta_{\bar{m}}(r) + \delta_{\bar{n}}(T - r) + 2 + l}} Y(u, x)v \\
= & (\lfloor n \rfloor + \delta_{\bar{n}}(T - r) + 1 - r/T + l) \operatorname{Res}_x \frac{e^{x(\delta_{\bar{m}}(r) + \lfloor m \rfloor + r/T)}}{(e^x - 1)^{\lfloor m\rfloor+\lfloor n \rfloor + \delta_{\bar{m}}(r) + \delta_{\bar{n}}(T - r) + 1 + l}} Y(u, x)v \\
& + (\lfloor m\rfloor+\lfloor n \rfloor + \delta_{\bar{m}}(r) + \delta_{\bar{n}}(T - r) + 1 + l) \operatorname{Res}_x \frac{e^{x(\delta_{\bar{m}}(r) + \lfloor m \rfloor + r/T)}}{(e^x - 1)^{\lfloor m\rfloor+\lfloor n \rfloor + \delta_{\bar{m}}(r) + \delta_{\bar{n}}(T - r) + 2 + l}} Y(u, x)v.
\end{align*}
By the induction hypothesis, the first term on the right-hand side belongs to $\tilde{O}_{g,n,m}^\prime(V)$. Since the left-hand side also lies in $\tilde{O}_{g,n,m}^\prime(V)$ (as $\mathcal{D}u \in V^r$), it follows that the second term must likewise belong to $\tilde{O}_{g,n,m}^\prime(V)$. This completes the inductive step and establishes the lemma.
\end{proof}
\begin{lem}\label{lem3.3}
The following statements hold:

{\rm(1)} For any $u, v \in V$, we have
    \[
    Y(u, x)v \equiv e^{x(n - m)}\, Y(v, -x)u  \bmod L_{n,m}(V).
    \]

{\rm(2)} If $r \not\equiv \bar{m} - \bar{n} \pmod{T}$, then $V^r \subseteq \tilde{O}_{g,n,m}^\prime(V)$.

{\rm(3)} Let $u \in V^r$ and $v \in V^s$. Suppose that
    $
    \bar{p} - \bar{n} \equiv r \pmod{T},  \bar{m} - \bar{p} \equiv s \pmod{T},
    $
    and that $m + n - p \geq 0$. Then
    \[
    u \bullet_{g, m, p}^{\,n} v - v \bullet_{g, m, m+n-p}^{\,n} u - \operatorname{Res}_x (1+x)^{-1 + p - n}\, Y(u, \log(1+x)) v \in L_{n,m}(V).
    \]
\end{lem}

\begin{proof}
    (1) By definition, $(\mathcal{D} + m - n)u \in L_{n,m}(V)$. Combining this with the skew-symmetry identity~\eqref{skew_sym} yields
    \[
    Y(u, x)v = e^{x\mathcal{D}} Y(v, -x)u \equiv e^{x(n - m)} Y(v, -x)u  \bmod L_{n,m}(V),
    \]
    as required.

    (2) Let $u \in V^r$ with $r \not\equiv \bar{m} - \bar{n} \pmod{T}$. Then
    \begin{align*}
        u \diamond_{g,m}^{\,n} \mathbf{1}
        &= \operatorname{Res}_{y} \frac{(1+y)^{-1 + \delta_{\bar{m}}(r) + \lfloor m \rfloor + r/T}}{y^{\lfloor m \rfloor + \lfloor n \rfloor + \delta_{\bar{m}}(r) + \delta_{\bar{n}}(T - r) + 1}}\, Y(u, \log(1+y)) \mathbf{1} \\
        &\equiv \operatorname{Res}_{y} \frac{(1+y)^{\delta_{\bar{n}}(r) + \lfloor n \rfloor + r/T - 1 + n - m}}{y^{2\lfloor n \rfloor + \delta_{\bar{n}}(r) + \delta_{\bar{n}}(T - r) + 1}}\, u 
         \bmod \tilde{O}_{g,n,m}^\prime(V) \qquad \text{(by part (1))} \\
        &= \binom{\delta_{\bar{n}}(r) + \lfloor n \rfloor + r/T - 1 + n - m}{\,2\lfloor n \rfloor + \delta_{\bar{n}}(r) + \delta_{\bar{n}}(T - r)\,} u.
    \end{align*}
    Since $r \not\equiv \bar{m} - \bar{n} \pmod{T}$, the binomial coefficient above is nonzero. On the other hand, by definition we have $u \diamond_{g,m}^{\,n} \mathbf{1} \in \tilde{O}_{g,n,m}^\prime(V)$. It follows that $u \in \tilde{O}_{g,n,m}^\prime(V)$, which completes the proof of part~(2).

(3) From the assumptions $\bar{p} - \bar{n} \equiv r \pmod{T}$ and $\bar{m} - \bar{p} \equiv s \pmod{T}$, one readily deduces that
$
-1 + \delta_{\bar{m}}(s) + \delta_{\bar{p}}(T - s) = 0,
-1 + \delta_{\bar{m}}(r) + \delta_{\bar{n}}(T - r) = \varepsilon,
$
where
\[
\varepsilon =
\begin{cases}
\phantom{-}1, & \text{if } \bar{m} + \bar{n} - \bar{p} \geq T, \\
\phantom{-}0, & \text{if } 0 \leq \bar{m} + \bar{n} - \bar{p} < T, \\
-1,           & \text{if } \bar{m} + \bar{n} - \bar{p} < 0.
\end{cases}
\]
Hence,
\begin{align*}
&\quad\, v \bullet_{g,\, m,\, m+n-p}^{\,n} u \\
&= \sum_{i=0}^{\lfloor n \rfloor + \lfloor m \rfloor - \lfloor p \rfloor + \varepsilon}
(-1)^i \binom{\lfloor p \rfloor + i}{i}
\operatorname{Res}_z \frac{(1+z)^{-1 + \lfloor m \rfloor + \delta_{\bar{m}}(s) + s/T}}{z^{\lfloor p \rfloor + i + 1}}
\, Y(v, \log(1+z))\, u \\
&\equiv \sum_{i=0}^{\lfloor n \rfloor + \lfloor m \rfloor - \lfloor p \rfloor + \varepsilon}
(-1)^i \binom{\lfloor p \rfloor + i}{i}
\operatorname{Res}_z \frac{(1+z)^{-1 + \lfloor n \rfloor + \delta_{\bar{m}}(s) + (s - \bar{m} + \bar{n})/T}}{z^{\lfloor p \rfloor + i + 1}} \\
&\qquad\qquad\qquad\qquad\qquad\qquad
\times Y(u, -\log(1+z))\, v
\bmod L_{n,m}(V) \quad\quad  \text{(by part (1))} \\
&= \sum_{i=0}^{\lfloor n \rfloor + \lfloor m \rfloor - \lfloor p \rfloor + \varepsilon}
(-1)^{\lfloor p \rfloor} \binom{\lfloor p \rfloor + i}{i}
\operatorname{Res}_z \frac{(1+z)^{-1 + p-n + i}}{z^{\lfloor p \rfloor + i + 1}}
\, Y(u, \log(1+z))\, v,
\end{align*}
where the last equality follows from $(\bar{p} + s - \bar{m})/T = \delta_{\bar{p}}(T - s)$.
Since $(\bar{n} + r - \bar{p})/T = \delta_{\bar{n}}(T - r)$, we obtain
\begin{align*}
    u \bullet_{g,\, m,\, p}^{\,n} v - v \bullet_{g,\, m,\, m+n-p}^{\,n} u 
\equiv  \operatorname{Res}_z A(z)\,
(1+z)^{-1 + p - n}\, Y(u, \log(1+z))\, v
\bmod L_{n,m}(V), 
\end{align*}
where
\begin{align*}
A(z)
= \sum_{i=0}^{\lfloor p \rfloor}
(-1)^i \binom{k + \varepsilon + i}{i}
\frac{(1+z)^{k + \varepsilon + 1}}{z^{k + \varepsilon + i + 1}}  - \sum_{i=0}^{k + \varepsilon}
(-1)^{\lfloor p \rfloor} \binom{\lfloor p \rfloor + i}{i}
\frac{(1+z)^i}{z^{\lfloor p \rfloor + i + 1}}
\end{align*}
 and $k=\lfloor n \rfloor + \lfloor m \rfloor - \lfloor p \rfloor$. The lemma now follows from \cite[Proposition 5.1]{DJ1}.
\end{proof}

\begin{cor}\label{cor3.4}
The following statements hold:

{\rm(1)} Let $u \in V^r$ and $v \in V^s$. If $\bar{m} - \bar{p} \not\equiv s \pmod{T}$, then 
    $
    u \bullet_{g, m, p}^{\,n} v \in \tilde{O}_{g, n, m}^{\prime}(V).
    $

{\rm(2)} Let $v \in V^s$,
if $\bar{m} - \bar{n} \equiv s \pmod{T}$, then $$v \bullet_{g, m}^{\,n} \mathbf{1} - v \in L_{n,m}(V).
$$
Moreover,
    $
    v \bullet_{g, m}^{\,n} \mathbf{1} - v \in \tilde{O}_{g, n, m}^{\prime}(V)
    $
    for any $v\in V.$
\end{cor}

\begin{proof}
(1) By definition, the product $u \bullet_{g, m, p}^{\,n} v$ is nonzero only if $\bar{p} - \bar{n} \equiv r \pmod{T}$.  
If this condition fails, then $u \bullet_{g, m, p}^{\,n} v = 0 \in \tilde{O}_{g, n, m}^{\prime}(V)$.  
If it holds, then $u \bullet_{g, m, p}^{\,n} v \in V^{r+s}$. The assumption $\bar{m} - \bar{p} \not\equiv s \pmod{T}$ implies
\[
\bar{m} - \bar{n} = (\bar{m} - \bar{p}) + (\bar{p} - \bar{n}) \not\equiv s + r \pmod{T}.
\]
Hence $r + s \not\equiv \bar{m} - \bar{n} \pmod{T}$, and by Lemma~\ref{lem3.3}(2) we conclude that $u \bullet_{g, m, p}^{\,n} v \in \tilde{O}_{g, n, m}^{\prime}(V)$.

(2) Applying Lemma~\ref{lem3.3}(3) with $u = \mathbf{1} \in V^0$ and $p=n$, and using the identity $\mathbf{1} \bar{\bullet}_{g, m}^{\,n} v = v$, we obtain
\[
v \bullet_{g, m}^{\,n} \mathbf{1} - v = v \bullet_{g, m}^{\,n} \mathbf{1} - \mathbf{1} \bar{\bullet}_{g, m}^{\,n} v \in L_{n,m}(V).
\]
Since $L_{n,m}(V) \subseteq \tilde{O}_{g, n, m}^{\prime}(V)$ by definition, it follows that $v \bullet_{g, m}^{\,n} \mathbf{1} - v \in \tilde{O}_{g, n, m}^{\prime}(V)$.  
Moreover, if $\bar{m} - \bar{n} \not\equiv s \pmod{T}$, then by part (2) of Lemma~\ref{lem3.3}, $u \in \tilde{O}_{g, n, m}^{\prime}(V)$, and since $u \bullet_{g, m}^{\,n} \mathbf{1} \in \tilde{O}_{g, n, m}^{\prime}(V)$ as well, their difference lies in $\tilde{O}_{g, n, m}^{\prime}(V)$. 
\end{proof}
\begin{lem}\label{lem3.5}
We have
\[
\tilde{O}_{g, n, m}^{\prime}(V) \,\bullet_{g, m}^{\,n}\, V \subseteq \tilde{O}_{g, n, m}^{\prime}(V),
\qquad V \,\bar{\bullet}_{g, m}^{\,n}\, \tilde{O}_{g, n, m}^{\prime}(V) \subseteq \tilde{O}_{g, n, m}^{\prime}(V).
\]
\end{lem}

\begin{proof}

The proof is analogous to that of \cite[Lemma~3.3]{Shun1}. 
We consider the expressions 
\[
((\mathcal{D} + m - n) u) \bullet_{g, m}^n v
\quad\text{and}\quad
v \,\bar{\bullet}_{g, m}^n\, ((\mathcal{D} + m - n) u).
\]
Assume that $u \in V^r$ and $v \in V^s$, where $\bm - \bn \equiv r \pmod{T}$ and $s = 0$, by the definition of $\bullet_{g, m}^n$, $\bar{\bullet}_{g, m}^n$ and Corollary~\ref{cor3.4}(1). 
Under these assumptions, we have
$
-1 + \delta_{\bm}(r) + \delta_{\bn}(T - r) = 0
$ and $(\bar{n}+r-\bar{m})/T=\delta_{\bar{n}}(T-r)$.
Then
\begin{align*}
    &((\D+m-n) u) \bullet_{g, m}^n v\\
    =&\sum_{i=0}^{\lfloor m\rfloor}(-1)^{i}\binom{\lfloor n\rfloor+i}{i} \operatorname{Res}_{y} \frac{(1+y)^{-1+\lfloor m\rfloor+\delta_{\bar{m}}(r)+r / T}}{y^{\lfloor n\rfloor+i+1}}
Y(\D u, \log(1+y)) v+(m-n)u\bullet_{g,m}^nv\\
=&\sum_{i=0}^{\lfloor m\rfloor}(-1)^{i}\binom{\lfloor n\rfloor+i}{i} \operatorname{Res}_{y} \frac{(1+y)^{\lfloor m\rfloor+\delta_{\bar{m}}(r)+r / T}}{y^{\lfloor n\rfloor+i+1}}\frac{d}{dy}
(Y( u, \log(1+y)) v)+(m-n)u\bullet_{g,m}^nv\\
=&\sum_{i=0}^{\lfloor m\rfloor}(-1)^{i+1}\binom{\lfloor n\rfloor+i}{i} \operatorname{Res}_{y}\frac{d}{dy} \left(\frac{(1+y)^{\lfloor m\rfloor+\delta_{\bar{m}}(r)+r / T}}{y^{\lfloor n\rfloor+i+1}}\right)
\\
&\cdot Y( u, \log(1+y)) v)+(m-n)u\bullet_{g,m}^nv\\
=&\sum_{i=0}^{\lfloor m\rfloor}(-1)^{i}\binom{\lfloor n\rfloor+i}{i} \operatorname{Res}_{y}
\frac{(\xn+i+1+iy)(1+y)^{\lfloor m\rfloor+\delta_{\bar{m}}(r)+r / T-1}}{y^{\xn+i+2}}Y( u, \log(1+y)) v).
\end{align*}
It is straightforward to show that
\begin{align*}
    &\sum_{i=0}^{\lfloor m\rfloor}(-1)^{i}\binom{\lfloor n\rfloor+i}{i}  
\frac{\xn+i+1+iy}{y^{\xn+i+2}}\\
=& \sum_{i=0}^{\lfloor m\rfloor}(-1)^{i}\binom{\lfloor n\rfloor+i}{\xqz{n}}  
\frac{i}{y^{\xn+i+1}}+ \sum_{i=0}^{\lfloor m\rfloor}(-1)^{i}\binom{\lfloor n\rfloor+i+1}{\xn}  
\frac{i+1}{y^{\xn+i+2}}\\
=& \sum_{i=0}^{\lfloor m\rfloor-1}(-1)^{i+1}\binom{\lfloor n\rfloor+i+1}{\xqz{n}}  
\frac{i+1}{y^{\xn+i+2}}+ \sum_{i=0}^{\lfloor m\rfloor}(-1)^{i}\binom{\lfloor n\rfloor+i+1}{\xn}  
\frac{i+1}{y^{\xn+i+2}}\\
=&(-1)^{\xm}\binom{\xn+\xm+1}{\xn}\frac{\xm+1}{y^{\xn+\xm+2}}.
\end{align*}
Then $((\D+m-n)u)\bullet_{g,m}^nv=(-1)^{\xm}\binom{\xn+\xm+1}{\xn}(\xm+1)u\diamond_{g,m}^nv\in \tilde{O}_{g,n,m}^\prime(V).$ 
By Lemma~\ref{lem3.3}(3), then
\begin{align*}
    &-v\bar{\bullet}_{g,m}^n((\D+m-n)u)
    \equiv-((\D+m-n)u)\bullet_{g,m}^nv\\
    &+\Res_x(1+x)^{-1+m-n} Y((\D+m-n)u,\log(1+x))v\bmod \tilde{O}_{g,n,m}^\prime(V)\\
    &\equiv \Res_x(1+x)^{-1+m-n}Y((\D+m-n)u,\log(1+x))v=0.
\end{align*}
Let $u_i \in V^{r_i}$ for $i = 1, 2, 3$.  
By the definition of $\bullet_{g, m}^n$ and Corollary~\ref{cor3.4}(1), we have
$
\bigl(u_1 \diamond_{g, m}^n u_2\bigr) \bullet_{g, m}^n u_3 \in \tilde{O}_{g, n, m}^{\prime}(V),
$
whenever $\bm - \bn \not\equiv r_1 + r_2 \pmod{T}$ or $r_3 \neq 0$.  
Hence, we may assume that
$
\bm - \bn \equiv r_1 + r_2 \pmod{T}$ and  $r_3 = 0.
$
Define the functions
\[
f(r) = \xqz{m} + \delta_{\bar{m}}(r) + {r}/{T},
\quad
h(r) = \xqz{m} + \xqz{n} + \delta_{\bar{m}}(r) + \delta_{\bar{n}}(T - r) + 1.
\]
Choose $s \in \{0, 1, \dots, T-1\}$ such that $s \equiv r_1 + r_2 \pmod{T}$.  
Then the following identities hold:
\begin{align*}
    &\delta_{\bar{m}}(r_1) + \delta_{\bar{n}}(T - r_1) 
      = \delta_{\bar{m}}(r_2) + \delta_{\bar{n}}(T - r_2), \\
    &\delta_{\bar{m}}(s) + \delta_{\bar{n}}(T - r_1) 
      = \delta_{\bar{m}}(r_2) - {(s - r_1 - r_2)}/{T}, \\
    &  \delta_{\bar{m}}(r_2) - \delta_{\bar{n}}(T - r_1) + (\bar{n} + r_1 + r_2-\bar{m}  )/T=1.
\end{align*}
From these, we deduce the relations
\begin{equation}\label{zuheeq4}
    h(r_1) = h(r_2), \quad
    h(r_1) - f(r_1) + f(s) = f(r_2) + \xqz{n} + 1, \quad
    h(r_1) - f(r_1) +m-n= f(r_2).
\end{equation}
Then we have
\begin{align*}
    &(u_1\diamond_{g,m}^n u_2) \bullet_{g,m}^nu_3\\
    =&\sum_{i=0}^{\xqz{m}}(-1)^i\binom{i+\xqz{n}}{i}\Res_{x_2}\Res_{x_0}
\frac{e^{x_2f(s)}}{(e^{x_2}-1)^{\xqz{n}+i+1}}\frac{e^{x_0f(r_1)}}{(e^{x_0}-1)^{h(r_1)}}Y(Y(u_1,x_0)u_2,x_2)u_3\\
    =&\sum_{i=0}^{\xqz{m}}(-1)^i\binom{i+\xqz{n}}{i}\Res_{x_1}\Res_{x_2}\Res_{x_0}
\frac{e^{x_2f(s)}}{(e^{x_2}-1)^{\xqz{n}+i+1}}\frac{e^{x_0f(r_1)}}{(e^{x_0}-1)^{h(r_1)}}\\
    &\quad\cdot x_1^{-1}\delta\left(
    \frac{x_2+x_0}{x_1}\right) Y(Y(u_1,x_0)u_2,x_2)u_3\\
    =&\sum_{i=0}^{\xqz{m}}(-1)^i\binom{i+\xqz{n}}{i}\Res_{x_1}\Res_{x_2}\Res_{x_0}
\frac{e^{x_2f(s)}}{(e^{x_2}-1)^{\xqz{n}+i+1}}\frac{e^{x_0f(r_1)}}{(e^{x_0}-1)^{h(r_1)}}\\
&\quad\cdot x_0^{-1} \delta\left(\frac{x_1-x_2}{x_0}\right) Y\left(u_1, x_1\right) Y\left( u_2, x_2\right) u_3\\
&-\sum_{i=0}^{\xqz{m}}(-1)^i\binom{i+\xqz{n}}{i}\Res_{x_1}\Res_{x_2}\Res_{x_0}
\frac{e^{x_2f(s)}}{(e^{x_2}-1)^{\xqz{n}+i+1}}\frac{e^{x_0f(r_1)}}{(e^{x_0}-1)^{h(r_1)}}\\
&\quad\cdot x_0^{-1} \delta\left(\frac{x_2-x_1}{-x_0}\right) Y\left(u_2, x_2\right) Y\left(u_1, x_1\right) u_3\\
=&\sum_{i=0}^{\xqz{m}}(-1)^i\binom{i+\xqz{n}}{i}\Res_{x_1}\Res_{x_2}
\frac{e^{x_2f(s)}}{(e^{x_2}-1)^{\xqz{n}+i+1}}\frac{e^{(x_1-x_2)f(r_1)}}{(e^{(x_1-x_2)}-1)^{h(r_1)}}\\
&\quad\cdot  Y\left( u_1, x_1\right) Y\left(u_2, x_2\right) u_3\\
&-\sum_{i=0}^{\xqz{m}}(-1)^i\binom{i+\xqz{n}}{i}\Res_{x_1}\Res_{x_2}
\frac{e^{x_2f(s)}}{(e^{x_2}-1)^{\xqz{n}+i+1}}\frac{e^{(-x_2+x_1)f(r_1)}}{(e^{(-x_2+x_1)}-1)^{h(r_1)}}\\
&\quad\cdot  Y\left(u_2, x_2\right) Y\left( u_1, x_1\right) u_3\\
=&\sum_{i=0}^{\xqz{m}}(-1)^i\binom{i+\xqz{n}}{i}\Res_{x_1}\Res_{x_2}
\frac{e^{x_2f(s)}}{(e^{x_2}-1)^{\xqz{n}+i+1}}\frac{e^{(x_1-x_2)f(r_1)}}{(e^{(x_1-x_2)}-e^{-x_2}+e^{-x_2}-1)^{h(r_1)}}\\
&\quad\cdot  Y\left( u_1, x_1\right) Y\left(u_2, x_2\right) u_3\\
&-\sum_{i=0}^{\xqz{m}}(-1)^i\binom{i+\xqz{n}}{i}\Res_{x_1}\Res_{x_2}
\frac{e^{x_2f(s)}}{(e^{x_2}-1)^{\xqz{n}+i+1}}\frac{e^{(-x_2+x_1)f(r_1)}}{(e^{(-x_2+x_1)}-e^{x_1}+e^{x_1}-1)^{h(r_1)}} \\
&\quad\cdot Y\left(u_2, x_2\right) Y\left( u_1, x_1\right) u_3\\
=&\sum_{i=0}^{\xqz{m}}\sum_{k\geq0}(-1)^{i+k}\binom{i+\xqz{n}}{i}\binom{-h(r_1)}{k}\Res_{x_1}\Res_{x_2}\frac{e^{x_1f(r_1)}}{(e^{x_1}-1)^{h(r_1)+k}}\\
&\quad\cdot\frac{e^{x_2(f(s)+h(r_1)-f(r_1))}}{(e^{x_2}-1)^{\xqz{n}+i+1-k}} Y\left(u_1, x_1\right) Y\left( u_2, x_2\right) u_3\\
&-\sum_{i=0}^{\xqz{m}}\sum_{k\geq0}(-1)^{i+h(r_1)+k}\binom{i+\xqz{n}}{i}\binom{-h(r_1)}{k}\Res_{x_1}\Res_{x_2}
\frac{e^{x_2(f(r_2)+\xqz{n}+1+k)}}{(e^{x_2}-1)^{h(r_2)+\xqz{n}+i+k+1}}\\
&\quad\cdot e^{x_1(f(r_1)-h(r_1)-k)}(e^{x_1}-1)^k
Y\left(u_2, x_2\right) Y\left( u_1, x_1\right) u_3\quad\quad\quad\quad\quad\quad\quad\quad\quad\quad\text{(by  \eqref{zuheeq4})}\\
\equiv&\,0\bmod \tilde{O}_{g,n,m}^\prime(V).\quad\quad\quad\quad\quad\quad\quad\quad\quad\quad\quad\quad\quad\quad\quad\quad\quad\quad\quad\quad\quad\quad\text{(by Lemma~\ref{prop1})}
\end{align*}
Let $u_i \in V^{r_i}$ for $i = 1, 2, 3$.  
By the definition of $\bar{\bullet}_{g, m}^n$ and Corollary~\ref{cor3.4}(1), we have
$
u_3 \,\bar{\bullet}_{g, m}^n\, \bigl(u_1 \diamond_{g, m}^n u_2\bigr) \in \tilde{O}_{g, n, m}^{\prime}(V),
$
whenever $\bm - \bn \not\equiv r_1 + r_2 \pmod{T}$ or $r_3 \neq 0$.  
Therefore, we may assume that
$
\bm - \bn \equiv r_1 + r_2 \pmod{T}$ and  $r_3 = 0.
$
Under this assumption, we obtain
\begin{align*}
    &-u_3 \bar{\bullet}_{g,m}^n(u_1\diamond_{g,m}^n u_2)\\
    \equiv& \operatorname{Res}_{x_2}\operatorname{Res}_{x_0}\frac{e^{x_0f(r_1)}e^{x_2(m-n)}}{(e^{x_0}-1)^{h(r_1)}} Y\left( Y(u_1,x_0)u_2, x_2\right) u_3  \bmod \tilde{O}_{g, n,m}^\prime(V)\quad\quad \text{(by Lemma~ \ref{lem3.3}(3))} \\
    = &\operatorname{Res}_{x_2}\operatorname{Res}_{x_0}\operatorname{Res}_{x_1}\frac{e^{x_0f(r_1)}e^{x_2(m-n)}}{(e^{x_0}-1)^{h(r_1)}} x_1^{-1}\delta\left(\frac{x_2+x_0}{x_1}\right)Y\left( Y(u_1,x_0)u_2, x_2\right) u_3\\
    =&\operatorname{Res}_{x_2}\operatorname{Res}_{x_0}\operatorname{Res}_{x_1}\frac{e^{x_0f(r_1)}e^{x_2(m-n)}}{(e^{x_0}-1)^{h(r_1)}} x_0^{-1}\delta\left(\frac{x_1-x_2}{x_0}\right)Y(u_1,x_1)Y\left( u_2, x_2\right) u_3\\
    &\quad-\operatorname{Res}_{x_2}\operatorname{Res}_{x_0}\operatorname{Res}_{x_1}\frac{e^{x_0f(r_1)}e^{x_2(m-n)}}{(e^{x_0}-1)^{h(r_1)}} x_0^{-1}\delta\left(\frac{-x_2+x_1}{x_0}\right)Y\left( u_2, x_2\right) Y(u_1,x_1)u_3\\
    =& \operatorname{Res}_{x_2}\operatorname{Res}_{x_1}\frac{e^{(x_1-x_2)f(r_1)}e^{x_2(m-n)}}{(e^{(x_1-x_2)}-1)^{h(r_1)}} Y\left(u_1, x_1\right) Y\left( u_2, x_2\right) u_3\\
    &\quad-\operatorname{Res}_{x_2}\operatorname{Res}_{x_1}\frac{e^{(x_1-x_2)f(r_1)}e^{x_2(m-n)}}{(e^{(-x_2+x_1)}-1)^{h(r_1)}} Y\left(u_2, x_2\right) Y\left( u_1, x_1\right) u_3\\
    =& \operatorname{Res}_{x_2}\operatorname{Res}_{x_1}\frac{e^{(x_1-x_2)f(r_1)}e^{x_2(m-n)}}{(e^{(x_1-x_2)}-e^{-x_2}+e^{-x_2}-1)^{h(r_1)}} Y\left(u_1, x_1\right) Y\left( u_2, x_2\right) u_3\\
    &\quad-\operatorname{Res}_{x_2}\operatorname{Res}_{x_1}\frac{e^{(x_1-x_2)f(r_1)}e^{x_2(m-n)}}{(e^{(-x_2+x_1)}-e^{x_1}+e^{x_1}-1)^{h(r_1)}} Y\left(u_2, x_2\right) Y\left( u_1, x_1\right) u_3\\
    =&\sum_{k\geq0}(-1)^k\binom{-h(r_1)}{k}\Res_{x_1}\Res_{x_2}\frac{e^{x_1f(r_1)}}{(e^{x_1}-1)^{h(r_1)+k}}e^{x_2(h(r_1)-f(r_1)+m-n)}\\
    &\cdot(e^{x_2}-1)^k Y\left(u_1, x_1\right) Y\left( u_2, x_2\right) u_3\\
&-\sum_{k\geq0}(-1)^{h(r_1)+k}\binom{-h(r_1)}{k}\Res_{x_1}\Res_{x_2}
\frac{e^{x_2(f(r_2)+k)}}{(e^{x_2}-1)^{h(r_2)+k}}e^{x_1(f(r_1)-h(r_1)-k)}\\
&\cdot(e^{x_1}-1)^k Y\left(u_2, x_2\right) Y\left( u_1, x_1\right) u_3\quad\quad\quad\quad\quad\quad\quad\quad\quad\quad\quad\quad\quad\quad\quad\quad\quad\quad\text{(by \eqref{zuheeq4})}\\
\equiv&\,0\bmod \tilde{O}_{g,n,m}^\prime(V). \quad\quad\quad\quad\quad\quad\quad\quad\quad\quad\quad\quad\quad\quad\quad\quad\quad\quad\quad\quad\quad\quad\quad\text{(by  Lemma~\ref{prop1})}
\end{align*}

\end{proof}

\begin{lem}
    For any $u_1, u_2, u_3 \in V$, we have
    \[
        \bigl(u_1 \,\bar{\bullet}_{g, m}^n\, u_2\bigr) \bullet_{g, m}^n u_3
        \;-\;
        u_1 \,\bar{\bullet}_{g, m}^n\, \bigl(u_2 \bullet_{g, m}^n u_3\bigr)
        \;\in\; \tilde{O}_{g, n, m}^{\prime}(V).
    \]
\end{lem}
\begin{proof}
    Let $u_i \in V^{r_i}$ for $i = 1, 2, 3$.  
By the definition of $\bullet_{g, m, p}^n$ and Lemma~\ref{lem3.5}, we have
\[
\bigl(u_1 \,\bar{\bullet}_{g, m}^n\, u_2\bigr) \bullet_{g, m}^n u_3
\;-\;
u_1 \,\bar{\bullet}_{g, m}^n\, \bigl(u_2 \bullet_{g, m}^n u_3\bigr)
\;\in\; \tilde{O}_{g, n, m}^{\prime}(V),
\]
whenever $r_1 \neq 0$, $\bar{m} - \bar{n} \not\equiv r_2 \pmod{T}$, or $r_3 \neq 0$.  
Therefore, we may assume that
$
r_1 = r_3 = 0$ and $\bar{m} - \bar{n} \equiv r_2 \pmod{T}.
$
Under this assumption, we proceed as follows.
    \begin{align*}
    &(u_1\bar{\bullet}_{g,m}^nu_2)\bullet_{g,m}^nu_3\\
        =&\sum_{m_1=0}^{\xqz{n}}\sum_{m_2=0}^{\xqz{m}}(-1)^{m_1+m_2}\binom{m_1+\xqz{m}}{\xqz{m}}\binom{m_2+\xqz{n}}{\xqz{n}}\\
&\cdot\Res_{x_2}\Res_{x_0}\frac{e^{x_2f(r_2)}}{(e^{x_2}-1)^{\xqz{n}+m_2+1}}\frac{e^{x_0(1+\xqz{m})}}{(e^{x_0}-1)^{\xqz{m}+m_1+1}} Y(Y(u_1,x_0)u_2,x_2)u_3\\
        =&\sum_{m_1=0}^{\xqz{n}}\sum_{m_2=0}^{\xqz{m}}(-1)^{m_1+m_2}\binom{m_1+\xqz{m}}{\xqz{m}}\binom{m_2+\xqz{n}}{\xqz{n}}\\
        &\cdot\Res_{x_2}\Res_{x_1}\frac{e^{x_2f(r_2)}}{(e^{x_2}-1)^{\xqz{n}+m_2+1}}\frac{e^{(x_1-x_2)(1+\xqz{m})}}{(e^{(x_1-x_2)}-1)^{\xqz{m}+m_1+1}}Y(u_1,x_1)Y(u_2,x_2)u_3\\
        &-\sum_{m_1=0}^{\xqz{n}}\sum_{m_2=0}^{\xqz{m}}(-1)^{m_1+m_2}\binom{m_1+\xqz{m}}{\xqz{m}}\binom{m_2+\xqz{n}}{\xqz{n}}\\
        &\cdot\Res_{x_2}\Res_{x_1}\frac{e^{x_2f(r_2)}}{(e^{x_2}-1)^{\xqz{n}+m_2+1}}\frac{e^{(x_1-x_2)(1+\xqz{m})}}{(e^{(-x_2+x_1)}-1)^{\xqz{m}+m_1+1}} Y(u_2,x_2)Y(u_1,x_1)u_3\\
        =&\sum_{m_1=0}^{\xqz{n}}\sum_{m_2=0}^{\xqz{m}}\sum_{k\geq 0}(-1)^{m_1+m_2+k}\binom{m_1+\xqz{m}}{\xqz{m}}\binom{m_2+\xqz{n}}{\xqz{n}}\binom{-\xqz{m}-m_1-1}{k}\\
        &\cdot\Res_{x_2}\Res_{x_1}\frac{e^{x_1(1+\xqz{m})}}{(e^{x_1}-1)^{\xqz{m}+m_1+1+k}}\frac{e^{x_2(f(r_2)+m_1)}}{(e^{x_2}-1)^{\xqz{n}+m_2+1-k}}Y(u_1,x_1)Y(u_2,x_2)u_3\\
&-\sum_{m_1=0}^{\xqz{n}}\sum_{m_2=0}^{\xqz{m}}\sum_{k\geq 0}(-1)^{m_2+\xqz{m}+k+1}\binom{m_1+\xqz{m}}{\xqz{m}}\binom{m_2+\xqz{n}}{\xqz{n}}\binom{-\xqz{m}-m_1-1}{k}\\
&\cdot\Res_{x_2}\Res_{x_1}\frac{e^{x_2(f(r_2)+m_1+k)}}{(e^{x_2}-1)^{\xqz{m}+\xqz{n}+m_2+m_1+k+2}}\frac{e^{x_1(-m_1-k)}}{(e^{x_1}-1)^{-k}}Y(u_2,x_2)Y(u_1,x_1)u_3\\
\equiv&\sum_{m_1=0}^{\xqz{n}}\sum_{m_2=0}^{\xqz{m}}\sum_{k\geq 0}(-1)^{m_1+m_2+k}\binom{m_1+\xqz{m}}{\xqz{m}}\binom{m_2+\xqz{n}}{\xqz{n}}\binom{-\xqz{m}-m_1-1}{k}\\
        &\cdot\Res_{y_2}\Res_{y_1}\frac{(1+y_1)^{\xqz{m}}}{y_1^{\xqz{m}+m_1+1+k}}\frac{{(1+y_2)}^{f(r_2)+m_1-1}}{y_2^{\xqz{n}+m_2+1-k}}\\
        &\cdot Y(u_1,\log(1+y_1))Y(u_2,\log(1+y_2))u_3\bmod \tilde{O}_{g,n,m}(V)\quad\quad\quad\quad\quad\quad\quad\text{(by Lemma~\ref{prop1})}\\
=&\sum_{m_1=0}^{\xqz{n}}\sum_{m_2=0}^{\xqz{m}}\sum_{k= 0}^{\xqz{n}-m_1}(-1)^{m_1+m_2+k}\binom{m_1+\xqz{m}}{\xqz{m}}\binom{m_2+\xqz{n}}{\xqz{n}}\binom{-\xqz{m}-m_1-1}{k}\\
        &\cdot\Res_{y_2}\Res_{y_1}\frac{(1+y_1)^{\xqz{m}}}{y_1^{\xqz{m}+m_1+1+k}}\frac{{(1+y_2)}^{f(r_2)+m_1-1}}{y_2^{\xqz{n}+m_2+1-k}} Y(u_1,\log(1+y_1))Y(u_2,\log(1+y_2))u_3\\
=&\sum_{m_1=0}^{\xqz{n}}\sum_{m_2=0}^{\xqz{m}}(-1)^{m_1+m_2}\binom{m_1+\xqz{m}}{\xqz{m}}\binom{m_2+\xqz{n}}{\xqz{n}}\cdot\Res_{y_2}\Res_{y_1}\frac{(1+y_1)^{\xqz{m}}}{y_1^{\xqz{m}+m_1+1}}\frac{{(1+y_2)}^{f(r_2)-1}}{y_2^{\xqz{n}+m_2+1}}\\
        &\quad\cdot Y(u_1,\log(1+y_1))Y(u_2,\log(1+y_2))u_3\quad\quad\quad\quad\quad\quad\quad\quad\quad\quad\text{(by \cite[Proposition 5.2]{DJ1})}\\
        =&\,u_1\bar{\bullet}_{g,m}^n(u_2\bullet_{g,m}^nu_3).    \end{align*}
\end{proof}

Let $\tilde{O}_{g, n, m}^{\prime\prime}(V)$ be the linear span of all elements of the form
\[
u \bullet_{g, m, p_3}^{\,n} \left( 
\bigl(a \bullet_{g, p_1, p_2}^{\,p_3} b\bigr) \bullet_{g, m, p_1}^{\,p_3} c 
- a \bullet_{g, m, p_2}^{\,p_3} \bigl(b \bullet_{g, m, p_1}^{\,p_2} c\bigr)
\right),
\]
where $a, b, c, u \in V$ and $p_1, p_2, p_3 \in (1/T)\mathbb{N}$. 
Define
\[
\tilde{O}_{g, n, m}^{\prime\prime\prime}(V) 
:= \sum_{p_1, p_2 \in (1/T)\mathbb{N}} 
\left( V \bullet_{g, p_1, p_2}^{\,n} \tilde{O}_{g, p_2, p_1}^{\prime}(V) \right) \bullet_{g, m, p_1}^{\,n} V,
\]
and set
\[
\tilde{O}_{g, n, m}(V) 
:= \tilde{O}_{g, n, m}^{\prime}(V) + \tilde{O}_{g, n, m}^{\prime\prime}(V) + \tilde{O}_{g, n, m}^{\prime\prime\prime}(V).
\]

\begin{lem}\label{lem3.7}
For any $m, n, p \in (1/T)\mathbb{N}$, we have
\[
V \bullet_{g, m, p}^{\,n} \tilde{O}_{g, p, m}(V) \subseteq \tilde{O}_{g, n, m}(V),
\qquad
\tilde{O}_{g, n, p}(V) \bullet_{g, m, p}^{\,n} V \subseteq \tilde{O}_{g, n, m}(V).
\]
In particular,
\[
V \,\bar{\bullet}_{g, m}^{\,n} \tilde{O}_{g, n, m}(V) \subseteq \tilde{O}_{g, n, m}(V),
\qquad
\tilde{O}_{g, n, m}(V) \bullet_{g, m}^{\,n} V \subseteq \tilde{O}_{g, n, m}(V).
\]
\end{lem}

\begin{proof}
Recall that $\mathbf{1} \,\bar{\bullet}_{g, m}^{\,n} u = u$ for all $u \in V$. First, observe that
\begin{equation}\label{eq3.2}
\tilde{O}_{g, n, p}^{\prime}(V) \bullet_{g, m, p}^{\,n} V 
\subseteq \bigl(V \,\bar{\bullet}_{g, p}^{\,n} \tilde{O}_{g, n, p}^{\prime}(V)\bigr) \bullet_{g, m, p}^{\,n} V 
\subseteq \tilde{O}_{g, n, m}^{\prime\prime\prime}(V),
\end{equation}
and by definition,
\begin{equation}\label{eq3.3}
\bigl(a \bullet_{g, p_1, p_2}^{\,n} b\bigr) \bullet_{g, m, p_1}^{\,n} c 
- a \bullet_{g, m, p_2}^{\,n} \bigl(b \bullet_{g, m, p_1}^{\,p_2} c\bigr) 
\in \tilde{O}_{g, n, m}^{\prime\prime}(V)
\end{equation}
for all $a, b, c \in V$ and $p_1, p_2 \in (1/T)\mathbb{N}$. Consequently,
\begin{equation}\label{eq3.4}
V \bullet_{g, m, p_2}^{\,n} \bigl( \tilde{O}_{g, p_2, p_1}^{\prime}(V) \bullet_{g, m, p_1}^{\,p_2} V \bigr) 
\subseteq \tilde{O}_{g, n, m}(V).
\end{equation}
Using Corollary~\ref{cor3.4} and the definition of $\tilde{O}_{g, n, m}^{\prime\prime\prime}(V)$, we obtain
\begin{equation}\label{eq3.5}
V \bullet_{g, m, p}^{\,n} \tilde{O}_{g, p, m}^{\prime}(V) 
\subseteq \bigl(V \bullet_{g, m, p}^{\,n} \tilde{O}_{g, p, m}^{\prime}(V)\bigr) \bullet_{g, m}^{\,n} V 
+ \tilde{O}_{g, n, m}^{\prime}(V) 
\subseteq \tilde{O}_{g, n, m}(V).
\end{equation}
Combining \eqref{eq3.2} and \eqref{eq3.5}, it suffices to verify that
\begin{equation}\label{eq3.6}
V \bullet_{g, m, p}^{\,n} \Bigl( 
\tilde{O}_{g, p, m}^{\prime\prime}(V)+  \bigl(V \bullet_{g, p_1, p_2}^{\,p} \tilde{O}_{g, p_2, p_1}^{\prime}(V)\bigr) \bullet_{g, m, p_1}^{\,p} V 
\Bigr) 
\subseteq \tilde{O}_{g, n, m}(V)
\end{equation}
and
\begin{equation}\label{eq3.7}
\Bigl( 
\tilde{O}_{g, n, p}^{\prime\prime}(V) +\bigl( V \bullet_{g, p_1, p_2}^{\,n} \tilde{O}_{g, p_2, p_1}^{\prime}(V)\bigr) \bullet_{g, p, p_1}^{\,n} V 
\Bigr) \bullet_{g, m, p}^{\,n} V 
\subseteq \tilde{O}_{g, n, m}(V)
\end{equation}
for all $p_1, p_2, p \in (1/T)\mathbb{N}$. We prove only \eqref{eq3.6}; the proof of \eqref{eq3.7} is analogous.  Consider a typical generator of $\tilde{O}_{g, p, m}^{\prime\prime}(V)$:
\[
u \bullet_{g, m, p_3}^{\,p} \left( 
\bigl(a \bullet_{g, p_1, p_2}^{\,p_3} b\bigr) \bullet_{g, m, p_1}^{\,p_3} c 
- a \bullet_{g, m, p_2}^{\,p_3} \bigl(b \bullet_{g, m, p_1}^{\,p_2} c\bigr) 
\right).
\]
Then by \eqref{eq3.3}, we obtain
\begin{align*}
& v \bullet_{g, m, p}^{\,n} \Bigl( 
u \bullet_{g, m, p_3}^{\,p} \bigl( 
(a \bullet_{g, p_1, p_2}^{\,p_3} b) \bullet_{g, m, p_1}^{\,p_3} c 
- a \bullet_{g, m, p_2}^{\,p_3} (b \bullet_{g, m, p_1}^{\,p_2} c) 
\bigr) 
\Bigr) \\
\equiv &\bigl(v \bullet_{g, p_3, p}^{\,n} u\bigr) \bullet_{g, m, p_3}^{\,n} \Bigl( 
(a \bullet_{g, p_1, p_2}^{\,p_3} b) \bullet_{g, m, p_1}^{\,p_3} c 
- a \bullet_{g, m, p_2}^{\,p_3} (b \bullet_{g, m, p_1}^{\,p_2} c) 
\Bigr) \\
\equiv& \,0 \bmod \tilde{O}_{g, n, m}(V).
\end{align*}
Thus we have
$V \bullet_{g, m, p}^{\,n}  
\tilde{O}_{g, p, m}^{\prime\prime}(V) 
\subseteq \tilde{O}_{g, n, m}(V)$.
By the definition of $\tilde{O}_{g, n, m}(V)$ together with \eqref{eq3.3} and \eqref{eq3.4}, we have
\begin{align*}
& V \bullet_{g, m, p}^{\,n} \Bigl( 
\bigl(V \bullet_{g, p_1, p_2}^{\,p} \tilde{O}_{g, p_2, p_1}^{\prime}(V)\bigr) \bullet_{g, m, p_1}^{\,p} V 
\Bigr) \\
 \subseteq& V \bullet_{g, m, p}^{\,n} \Bigl( 
V \bullet_{g, m, p_2}^{\,p} \bigl( \tilde{O}_{g, p_2, p_1}^{\prime}(V) \bullet_{g, m, p_1}^{\,p_2} V \bigr) 
\Bigr) + \tilde{O}_{g, n, m}(V) \\
 \subseteq &\bigl(V \bullet_{g, p_2, p}^{\,n} V\bigr) \bullet_{g, m, p_2}^{\,n} 
\bigl( \tilde{O}_{g, p_2, p_1}^{\prime}(V) \bullet_{g, m, p_1}^{\,p_2} V \bigr) 
+ \tilde{O}_{g, n, m}(V) \\
 \subseteq & \tilde{O}_{g, n, m}(V).
\end{align*}
This completes the proof.
\end{proof}

We now define the quotient space
\[
\tilde{A}_{g, n, m}(V) := V / \tilde{O}_{g, n, m}(V).
\]
The following is the first main theorem of this paper.

\begin{thm}\label{thm3.8}
 $\tilde{A}_{g, n, m}(V)$ carries the structure of an 
$\tilde{A}_{g, n}(V)$--$\tilde{A}_{g, m}(V)$-bimodule, where the left action of 
$\tilde{A}_{g, n}(V)$ is given by $\bar{\bullet}_{g, m}^{\,n}$ and the right action of 
$\tilde{A}_{g, m}(V)$ is given by $\bullet_{g, m}^{\,n}$.
\end{thm}

Motivated by Dong-Jiang's conjecture in~\cite{DJ1,DJ2}, we  propose the following conjecture: 
\begin{conj}\label{conj7.8}
  For all $n, m \in (1/T)\mathbb{N}$, one has $\tilde{O}^{\prime}_{g,n,m}(V) = \tilde{O}_{g,n,m}(V).$ 
\end{conj}

\section{The functor $\tilde{\Omega}_n$}

Let $(W, Y_W)$ be a $g$-twisted $\phi$-coordinated $V$-module.  
For any $i \in (1/T)\Z$ and $u \in V$, define
$
o_i(u) = u_{-i-1} \in \End(W)$ and $o(u) = o_0(u).
$
For $n \in (1/T)\N$, set
\[
\tilde{\Omega}_n(W) = \left\{ w \in W \;\middle|\; o_m(v) w = 0 \text{ for all } v \in V \text{ and } m \in (1/T)\Z \text{ with } m < -n \right\}.
\]

\begin{lem}\label{lem4.1}
    Let $u \in V^r$, $v \in V^s$, and $m, p, n \in (1/T)\N$.  
    Then for any $w \in \tilde{\Omega}_m(W)$, we have
    \[
        o_{\,n - m}\!\bigl(u \bullet_{g, m, p}^{\,n} v\bigr) w
        \;=\;
        o_{\,n - p}(u)\, o_{\,p - m}(v) w.
    \]
\end{lem}
\begin{proof}
   If $\bm - \bp \not\equiv s \pmod{T}$ and $\bp - \bn \equiv r \pmod{T}$,  
then $o_{p - m}(v) = 0$ and 
$
o_{\,n - m}\!\bigl(u \bullet_{g, m, p}^{\,n} v\bigr) w = 0.
$
If $\bp - \bn \not\equiv r \pmod{T}$, then we have $u \bullet_{g, m, p}^{\,n} v = 0$ and $o_{n-p}(u)=0$.

Finally, we consider the case where
$
\bp - \bn \equiv r$  and  $\bm - \bp \equiv s \pmod{T}
$. Set
\[
f(r)=-1+\lfloor m\rfloor+\delta_{\bar{m}}(r)+r / T,\quad h(r)=\lfloor m\rfloor+\lfloor n\rfloor-\lfloor p\rfloor-1+\delta_{\bar{m}}(r)+\delta_{\bar{n}}(T-r).
\]
For any $w \in \tilde{\Omega}_m(W)$, we have
\begin{align*}
    &\quad o_{n-m}(u\bullet_{g,m,p}^nv)w\\
    &=\sum_{i=0}^{\lfloor p\rfloor}(-1)^{i}\binom{h(r)+i}{i} \Res_{x_2}\operatorname{Res}_{y_0}x_2^{-n+m-1} \frac{(1+y_0)^{f(r)}}{y_0^{h(r)+1+i}} Y_W(Y(u, \log(1 + y_0))v,x_2)w\\
    &=\sum_{i=0}^{\lfloor p\rfloor}(-1)^{i}\binom{h(r)+i}{i}\Res_{x_1}\Res_{x_2}\operatorname{Res}_{y_0}x_2^{-n+m-1} \frac{(1+y_0)^{f(r)}}{y_0^{h(r)+1+i}}\left(\frac{x_2(1+y_0)}{x_1}\right)^{-r / T}\\
    &
    \quad\quad\quad \times x_1^{-1} \delta\left(\frac{x_2(1+y_0)}{x_1}\right)\left(\frac{x_2(1+y_0)}{x_1}\right)^{r / T} Y_W(Y(u, \log(1 + y_0))v,x_2)w\\
    &=\sum_{i=0}^{\lfloor p\rfloor}(-1)^{i}\binom{h(r)+i}{i}\Res_{x_1}\Res_{x_2}\operatorname{Res}_{y_0}x_2^{-n+m-1} \frac{(1+y_0)^{f(r)}}{y_0^{h(r)+1+i}}\left(\frac{x_2(1+y_0)}{x_1}\right)^{-r / T}\\
    &
    \quad\quad\quad \times\left(\left(x_2 y_0\right)^{-1} \delta\left(\frac{x_1-x_2}{x_2 y_0}\right) Y_W\left(u, x_1\right) Y_W\left(v, x_2\right)\right.\\
    &\quad\quad\quad\quad\quad\left.-\left(x_2 y_0\right)^{-1} \delta\left(\frac{x_2-x_1}{-x_2 y_0}\right) Y_W\left(v, x_2\right) Y_W\left(u, x_1\right)\right)\\
    &=\sum_{i=0}^{\lfloor p\rfloor}(-1)^{i}\binom{h(r)+i}{i}\Res_{x_2}\Res_{x_1}\left(
    \frac{x_1^{f(r)}x_2^{m-1-(\bn+r)/T-\xp+\delta_{\bn}(T-r)+i}}{(x_1-x_2)^{h(r)+1+i}}Y_W(u,x_1)Y_W(v,x_2)w
    \right.\\
    &\left.\quad\quad\quad\quad\quad\quad\quad\quad\quad\quad
    -\frac{x_1^{f(r)}x_2^{m-1-(\bn+r)/T-\xp+\delta_{\bn}(T-r)+i}}{(-x_2+x_1)^{h(r)+1+i}}Y_W(v,x_2)Y_W(u,x_1)w
    \right)\\
    &=\sum_{i=0}^{\lfloor p\rfloor}(-1)^{i}\binom{h(r)+i}{i}\Res_{x_2}\Res_{x_1}
    \frac{x_1^{f(r)}x_2^{-p+m-1+i}}{(x_1-x_2)^{h(r)+1+i}}Y_W(u,x_1)Y_W(v,x_2)w
   \\
    &=\sum_{i=0}^{\lfloor p\rfloor}\sum_{j\geq 0}(-1)^{i+j}\binom{h(r)+i}{i}\binom{-(h(r)+1+i)}{j}\\
    &
    \quad\quad\quad\times\Res_{x_2}\Res_{x_1}x_1^{-n+p-1-i-j}x_2^{-p+m-1+i+j}Y_W(u,x_1)Y_W(v,x_2)w\\
    &=\sum_{k=0}^{\lfloor p\rfloor}(-1)^{k}\sum_{i,j\geq 0, i+j=k}\binom{h(r)+i}{i}\binom{-(h(r)+1+i)}{j}\\
    &
    \quad\quad\quad\times\Res_{x_2}\Res_{x_1}x_1^{-n+p-1-k}x_2^{-p+m-1+k}Y_W(u,x_1)Y_W(v,x_2)w.
\end{align*}
It is enough to show that 
\[
\sum_{i=0}^k(-1)^i\binom{l+i}{i}\binom{l+k}{k-i}=0
\]
for any $l\in\Z$ and $k\in\Z_{\geq1}$, which follows from an easy calculation:
\begin{align*}
    \sum_{i=0}^k(-1)^i\binom{l+i}{i}\binom{l+k}{k-i}=\sum_{i=0}^k(-1)^i\binom{l+k}{k}\binom{k}{i}=\binom{l+k}{k}\sum_{i=0}^k(-1)^i\binom{k}{i}=0.
\end{align*}
This completes the proof.
\end{proof}

\begin{lem}\label{lem4.3}
    Let $m, n \in (1/T)\N$.  
    Then for every $u \in \tilde{O}_{g,n,m}(V)$, the operator $o_{\,n - m}(u)$ vanishes on $\tilde{\Omega}_m(W)$.  
    In particular, for every $u \in \tilde{O}_{g,n}(V)$, we have $o(u) = 0$ on $\tilde{\Omega}_n(W)$.
\end{lem}

\begin{proof}
For any $u \in V$, Lemma~\ref{daozi} yields
\begin{align*}
    o_{\,n - m}(\D u)
    &= \Res_x\, x^{-n + m - 1} Y_W(\D u, x)
     = \Res_x\, x^{-n + m} \frac{d}{dx} Y_W(u, x) \\
    &= -\Res_x\, \frac{d}{dx}\!\bigl(x^{-n + m}\bigr) Y_W(u, x)
     = -(m - n) \Res_x\, x^{-n + m - 1} Y_W(u, x) \\
    &= -(m - n)\, o_{\,n - m}(u).
\end{align*}
Consequently,
$
o_{\,n - m}\bigl((\D + m - n)u\bigr) = 0.
$

Now let $u \in V^r$ and $v \in V^s$.  
We first show that $o_{\,n - m}\!\bigl(u \diamond_{g, m}^{\,n} v\bigr) = 0$ on $\tilde{\Omega}_m(W)$.  
This is immediate when $\bm - \bn \not\equiv r + s \pmod{T}$.  
Thus, we may assume that $\bm - \bn \equiv r + s \pmod{T}$.  
For any $w \in \tilde{\Omega}_m(W)$, we have
\begin{align*}
    &\quad o_{n-m}(u\diamond_{g,m}^nv)w=(u\diamond_{g,m}^nv)_{-n+m-1}w\\
    &=\Res_{x_2}\operatorname{Res}_{y_0}x_2^{-n+m-1} \frac{(1+y_0)^{-1+\delta_{\bar{m}}(r)+\lfloor m\rfloor+r / T}}{y_0^{\lfloor m\rfloor+\lfloor n\rfloor+\delta_{\bar{m}}(r)+\delta_{\bar{n}}(T-r)+1}}  Y_W(Y(u, \log(1 + y_0))v,x_2)w\\
    &= \Res_{x_2}\Res_{x_1}\operatorname{Res}_{y_0}
    x_2^{-n+m-1} \frac{(1+y_0)^{-1+\delta_{\bar{m}}(r)+\lfloor m\rfloor+r / T}}{y_0^{\lfloor m\rfloor+\lfloor n\rfloor+\delta_{\bar{m}}(r)+\delta_{\bar{n}}(T-r)+1}} 
    \left(\frac{x_2(1+y_0)}{x_1}\right)^{-r / T}\\
    &
    \quad\quad\quad x_1^{-1} \delta\left(\frac{x_2(1+y_0)}{x_1}\right)\left(\frac{x_2(1+y_0)}{x_1}\right)^{r / T}
    Y_W(Y(u, \log(1 + y_0))v,x_2)w\\
    &=\Res_{x_2}\Res_{x_1}\operatorname{Res}_{y_0}
    x_2^{-n+m-1}\frac{(1+y_0)^{-1+\delta_{\bar{m}}(r)+\lfloor m\rfloor+r / T}}{y_0^{\lfloor m\rfloor+\lfloor n\rfloor+\delta_{\bar{m}}(r)+\delta_{\bar{n}}(T-r)+1}} 
    \left(\frac{x_2(1+y_0)}{x_1}\right)^{-r / T}\\
    &\quad\quad\quad\times\left(\left(x_2 y_0\right)^{-1} \delta\left(\frac{x_1-x_2}{x_2 y_0}\right) Y_W\left(u, x_1\right) Y_W\left(v, x_2\right)w\right.\\
    &\quad\quad\quad\quad\quad\left.-\left(x_2 y_0\right)^{-1} \delta\left(\frac{x_2-x_1}{-x_2 y_0}\right) Y_W\left(v, x_2\right) Y_W\left(u, x_1\right)w\right)\\
    &=\Res_{x_2}\Res_{x_1}\left(
  \frac{x_1^{-1+\delta_{\bar{m}}(r)+\lfloor m\rfloor+r / T}x_2^{m-(\bn+r)/T+\delta_{\bn}(T-r)}}{(x_1-x_2)^{\lfloor m\rfloor+\lfloor n\rfloor+\delta_{\bar{m}}(r)+\delta_{\bar{n}}(T-r)+1}} Y_W(u,x_1)Y_W(v,x_2)w
    \right.\\
    &\quad\quad\quad\quad\quad\quad\quad\left.
    - \frac{x_1^{-1+\delta_{\bar{m}}(r)+\lfloor m\rfloor+r / T}x_2^{m-(\bn+r)/T+\delta_{\bn}(T-r)}}{(-x_2+x_1)^{\lfloor m\rfloor+\lfloor n\rfloor+\delta_{\bar{m}}(r)+\delta_{\bar{n}}(T-r)+1}}Y_W(v,x_2)Y_W(u,x_1)w
    \right)=0.
\end{align*}
By applying Lemma~\ref{lem4.1}, we deduce that both \( o_{n-m}\bigl(\tilde{O}_{g,n,m}^{\prime\prime}(V)\bigr) \) and \( o_{n-m}\bigl(\tilde{O}_{g,n,m}^{\prime\prime\prime}(V)\bigr) \) vanish in \( \tilde{\Omega}_m(W) \).
\end{proof}

Next we associate to each $g$-twisted $\phi$-coordinated $V$-module an  $\tilde{A}_{g,n}(V)$-module structure.

\begin{thm}\label{thm4.4}
    Let $(W, Y_W)$ be a $g$-twisted $\phi$-coordinated $V$-module, and let $n \in (1/T)\mathbb{N}$.  
    Then $\tilde{\Omega}_n(W)$ carries the structure of an $\tilde{A}_{g,n}(V)$-module, 
    where the action of $v \in V$ is given by the operator $o(v)$.
\end{thm}

\begin{proof}
We first show that $\tilde{\Omega}_n(W)$ is closed under the action of $o(v)$ for every $v \in V$.  
Let $u \in V^r$, $v \in V$, and $w \in \tilde{\Omega}_n(W)$.  
By the commutator formula~\eqref{comm_formula}, we have
\[
o_p(u)\, o_q(v) w - o_q(v)\, o_p(u) w
= \sum_{j \geq 0} \frac{(-p)^j}{j!}\, o_{p+q}(u_j v) w,
\]
for all $p \in -r/T + \Z$ and $q \in (1/T)\Z$.  

If $p < -n$ and $q \leq 0$, then $o_p(u) w = 0$ by definition of $\tilde{\Omega}_n(W)$, and the right-hand side also vanishes because $p + q < -n$. Hence $o_p(u)\, o_q(v) w = 0$, which implies that $o_q(v) w \in \tilde{\Omega}_n(W)$ for all $q \leq 0$.  
In particular, taking $q = 0$, we obtain $o(v) w \in \tilde{\Omega}_n(W)$.

Now let $w \in \tilde{\Omega}_n(W)$.  
By Lemmas~\ref{lem4.1} and~\ref{lem4.3}, we have
$
o(u) w = 0$  for all   $u\in \tilde{O}_{g,n}(V)$ and
$o(u \bullet_{g,n} v) w = o(u)\, o(v) w$ for all   $u, v \in V.
$
Therefore, the assignment $v \mapsto o(v)$ defines a well-defined action of $\tilde{A}_{g,n}(V)$ on $\tilde{\Omega}_n(W)$, making it an $\tilde{A}_{g,n}(V)$-module.
\end{proof}

\begin{prop}\label{prop4.5}
Suppose that $M = \bigoplus_{m \in (1/T)\N} M(m)$ is a $(1/T)\N$-graded $g$-twisted $\phi$-coordinated $V$-module, and let $n \in (1/T)\N$. Then the following hold:

{\rm(1)} $\tilde{\Omega}_n(M) \supset \bigoplus_{n\geq i \in (1/T)\N } M(i)$.  
    If $M$ is simple, then $
    \tilde{\Omega}_n(M) = \bigoplus_{n\geq i \in (1/T)\N } M(i).
    $

{\rm(2)} If $M$ is simple and $M(m) \neq 0$ for some $m \in (1/T)\N$ with $m \leq n$, then $M(m)$ is an irreducible $\tilde{A}_{g,n}(V)$-module.
 
\end{prop}

\begin{proof}
It is straightforward to verify that $\tilde{\Omega}_n(M)$ is a graded subspace of $M$; namely,
\[
\tilde{\Omega}_n(M) = \bigoplus_{i \in (1/T)\N} \bigl( \tilde{\Omega}_n(M) \cap M(i) \bigr).
\]
Set $\tilde{\Omega}_n(i) := \tilde{\Omega}_n(M) \cap M(i)$.  
Clearly, $M(i) \subseteq \tilde{\Omega}_n(M)$ whenever $i \leq n$, since $o_p(v) M(i) \subseteq M(i + p)$ and $p < -n$ implies $i + p < 0$.

To prove part (1), it remains to show that $\tilde{\Omega}_n(i) = 0$ for all $i > n$.  
Assume, to the contrary, that $\tilde{\Omega}_n(i) \neq 0$ for some $i > n$, and choose a nonzero vector $w \in \tilde{\Omega}_n(i)$.  
By Proposition~\ref{prop2.8}, $M$ is generated by any nonzero vector; in particular,
\[
M = \operatorname{span}\bigl\{ u_p w \mid u \in V,\ p \in (1/T)\Z \bigr\}.
\]
However, since $w \in \tilde{\Omega}_n(M)$, we have $o_p(v) w  = 0$ for all $p < -n$.  
Consequently, $M$ has no components in degree less than $i-n >0$, implying in particular that $M(0) = 0$.  
Hence $\tilde{\Omega}_n(i) = 0$ for $i > n$, proving~(1).

For part (2), let $U \subseteq M(m)$ be a nonzero $\tilde{A}_{g,n}(V)$-submodule.  
Consider the subspace
\[
N := \operatorname{span}\bigl\{ u_p U \mid u \in V,\ p \in (1/T)\Z \bigr\} \subseteq M.
\]
Then $N$ is a $(1/T)\N$-graded $g$-twisted $\phi$-coordinated $V$-submodule of $M$, and its degree-$m$ component is $U$.  
Since $M$ is simple, we must have $N = M$, which forces $U = M(m)$.  
Therefore, $M(m)$ is an irreducible $\tilde{A}_{g,n}(V)$-module.
\end{proof}
\section{The Lie algebra $\hat{V}[g]$}

Let $V$ be a vertex algebra equipped with an automorphism $g$ of order $T$.  
Let $t$ be an indeterminate, and consider the tensor product
\[
\mathcal{L}(V) = V \otimes \C\!\left[t^{1/T}, t^{-1/T}\right].
\]
Following \cite{L2}, we endow $\C\!\left[t^{1/T}, t^{-1/T}\right]$ with the structure of a vertex algebra whose vertex operator is defined by
\[
Y(f(t), z) g(t) = \bigl(e^{z\,  (t\frac{d}{dt})} f(t)\bigr) g(t).
\]
Consequently, $\mathcal{L}(V)$ becomes the tensor product of two vertex algebras and hence itself a vertex algebra \cite{LL1}.
The automorphism $g$ extends naturally to a vertex algebra automorphism of $\mathcal{L}(V)$ via
\[
g\bigl(a \otimes t^m\bigr) = e^{2\pi i m}\, (g a \otimes t^m), \qquad a \in V,\ m \in (1/T)\Z.
\]
Denote by $\mathcal{L}(V, g)$ the subspace of $g$-invariants; it is a vertex subalgebra of $\mathcal{L}(V)$. Explicitly,
\[
\mathcal{L}(V, g) = \bigoplus_{r=0}^{T-1} V^r \otimes t^{r/T} \C[t, t^{-1}].
\]
The $\D$-operator on $\mathcal{L}(V, g)$ is given by
$
\D_{\mathcal{L}} = \D \otimes 1 + 1 \otimes t \frac{d}{dt}.
$
As shown in \cite{B1}, the quotient
\[
\hat{V}[g] := \mathcal{L}(V, g) / \D_{\mathcal{L}} \mathcal{L}(V, g)
\]
carries a natural Lie algebra structure with bracket
\[
[u + \D_{\mathcal{L}} \mathcal{L}(V, g),\, v + \D_{\mathcal{L}} \mathcal{L}(V, g)] 
= u_0 v + \D_{\mathcal{L}} \mathcal{L}(V, g),
\qquad u, v \in \mathcal{L}(V, g).
\]
For notational convenience, we denote by $a(q)$ the image of $a \otimes t^q \in \mathcal{L}(V, g)$ in $\hat{V}[g]$.  
A direct computation yields the following:

\begin{lem}\label{lem5.1}
Let $a \in V^r$, $b \in V^s$, and $m, n \in \Z$. Then

{\rm(1)}
    $
    \bigl[a\bigl(m + {r}/{T}\bigr),\, b\bigl(n + {s}/{T}\bigr)\bigr]
    = \sum_{i=0}^{\infty} \frac{(m + r/T)^i}{i!}\, (a_i b)(m + n + {(r + s)}/{T});
    $

{\rm(2)} $\mathbf{1}(0)$ lies in the center of $\hat{V}[g]$.
 
\end{lem}

We introduce a $(1/T)\Z$-grading on $\mathcal{L}(V, g)$ by setting
$
\deg(a \otimes t^n) = -n$ for $a \in V,\ n \in (1/T)\Z.
$
Since $\D_{\mathcal{L}}$ preserves this grading, $\D_{\mathcal{L}} \mathcal{L}(V, g)$ is a graded subspace, and thus $\hat{V}[g]$ inherits a natural $(1/T)\Z$-grading.  
Let $\hat{V}[g]_n$ denote the homogeneous component of degree $n$.  
By Lemma~\ref{lem5.1}, $\hat{V}[g]$ is a $(1/T)\Z$-graded Lie algebra, and we have the triangular decomposition
\[
\hat{V}[g] = \hat{V}[g]_{+} \oplus \hat{V}[g]_0 \oplus \hat{V}[g]_{-},
\]
where
$
\hat{V}[g]_{\pm} = \bigoplus_{0< n \in (1/T)\Z } \hat{V}[g]_{\pm n}.
$
Note that $\hat{V}[g]_0$ is spanned by elements of the form $a(0)$ with $a \in V^0$.  
From Lemma~\ref{lem5.1}(1), the Lie bracket is given by
$
[a(0), b(0)] = (a_0 b)(0)$ for $a, b \in V^0.
$
Consider the linear map $V^0 \to \hat{V}[g]_0$ sending $a$ to $a(0)$.
Its kernel is precisely $\D(V^0)$, and it induces an isomorphism of Lie algebras
$
V^0 / \D(V^0) \cong \hat{V}[g]_0,
$
where the Lie bracket on the quotient is defined by
\[
[a + \D(V^0),\, b + \D(V^0)] = a_0 b + \D(V^0), \qquad a, b \in V^0.
\]
From Lemma~\ref{lem3.3}, we obtain the following:

\begin{lem}\label{lem5.2}
The following statements hold:

{\rm(1)} If $r \neq 0$, then $V^r \subset \tilde{O}_{g,n}(V)$.

{\rm(2)} For any $u, v \in V^0$, we have the congruence
    \[
    u \bullet_{g,n} v - v \bullet_{g,n} u \equiv \Res_x\, Y(u, x)v \pmod{\tilde{O}_{g,n}(V)}.
    \]
 
\end{lem}

Combining the isomorphism $V^0 / \D(V^0) \cong \hat{V}[g]_0$ with Lemma~\ref{lem5.2}, we deduce:

\begin{lem}\label{lem5.3}
Let $\tilde{A}_{g,n}(V)_{\mathrm{Lie}}$ denote the Lie algebra obtained from the associative algebra $\tilde{A}_{g,n}(V)$ via the commutator bracket
$
[u, v] = u \bullet_{g,n} v - v \bullet_{g,n} u.
$
Then the assignment
\[
a(0) \longmapsto a + \tilde{O}_{g,n}(V)
\]
defines a surjective homomorphism of Lie algebras
$
\hat{V}[g]_0 \twoheadrightarrow \tilde{A}_{g,n}(V)_{\mathrm{Lie}}.
$
\end{lem}

\section{The functor $\tilde{L}$}\label{sec6}

Fix an $\tilde{A}_{g }(V)$-module $U$,  
then $U$ naturally becomes a module for the Lie algebra $\tilde{A}_{g }(V)_{\mathrm{Lie}}$ via the commutator bracket.  
By Lemma~\ref{lem5.3}, this action lifts to a $\hat{V}[g]_0$-module structure on $U$.  
We further extend this to a module for the subalgebra
$
P_0 = \bigoplus_{p \geq0  } \hat{V}[g]_{-p} 
$
by letting $\hat{V}[g]_{-p}$ act trivially for all $p > 0$.
Define the induced module
\[
\tilde{M}(U) = \operatorname{Ind}_{P_0}^{\hat{V}[g]}(U)
= U(\hat{V}[g]) \otimes_{U(P_0)} U.
\]
Assigning degree $0$ to $U$, the $(1/T)\Z$-grading on $\hat{V}[g]$ induces a $(1/T)\N$-grading on $\tilde{M}(U)$, making it a graded $\hat{V}[g]$-module.  
By the Poincaré–Birkhoff–Witt theorem,
$
\tilde{M}(U)(i) = U(\hat{V}[g])_{\,i  }\, U$ for any $i \in (1/T)\N.
$

Let $U^* = \operatorname{Hom}_{\C}(U, \C)$,  we extend $U^*$ to all of $\tilde{M}(U)$ by  
$
\langle u', w \rangle = 0$ for all  $u' \in U^*,\ w \in \tilde{M}(U)(i),\ i \neq 0.
$
We define, for $v \in V^r$, the vertex operator on $\tilde{M}(U)$ by
\[
Y_{\tilde{M}(U)}(v, z) = \sum_{m \in r/T + \Z} v(m)\, z^{-m}.
\]
In the following, we simply write $Y$ for $Y_{\tilde{M}(U)}$; this should cause no confusion.

\begin{lem}\label{lem6.1}
    For all $a \in V^r$, $b \in V$, $u' \in U^*$, $u \in \tilde{M}(U)$, and $X \in U(\hat{V}[g])$, there exists a positive integer $k$ such that
    \begin{align}
        &\left\langle u',\, (z_0 + 1)^{k + r/T}\, X\, Y \!\bigl(a, (z_0 + 1) z_2\bigr)\, Y (b, z_2)\, u \right\rangle \nonumber \\
         = &\left\langle u',\, (1 + z_0)^{k + r/T}\, X\, Y \!\bigl(Y(a, \log(1 + z_0))\, b, z_2\bigr)\, u \right\rangle. \label{eq6.5}
    \end{align}
\end{lem}

\begin{proof}

\noindent\textbf{Claim 1.} For all $a \in V^r$, $b \in V$, $u' \in U^*$, $u \in U$ and $i,j \in \N$,
\begin{align*}
&\Res_{z_0}\, z_0^{\,i} (z_0 + 1)^{-1+ \delta_{0}(r) + r/T + j}
   \left\langle u',\, Y\bigl(a, (z_0 + 1) z_2\bigr)\, Y(b, z_2)\, u \right\rangle \\
  =& \Res_{z_0}\, z_0^{\,i} (1 + z_0)^{-1+ \delta_{0}(r) + r/T + j}
   \left\langle u',\, Y\bigl(Y(a, \log(1 + z_0))\, b, z_2\bigr)\, u \right\rangle.
\end{align*}
Indeed, since $j \geq 0$, the mode
$
a\bigl(-1+ \delta_{0}(r) + r/T + j\bigr)
$
lies in $\bigoplus_{p > 0} \hat{V}[g]_{-p}$ and therefore annihilates $u$.  
Consequently, for all $i \in \N$ we have
\[
\Res_{z_1}\, (z_1 - z_2)^i\, z_1^{-1+ \delta_{0}(r) + r/T + j}\,
Y(b, z_2)\, Y(a, z_1)\, u = 0.
\]
From Lemma~\ref{lem5.1}, we have the commutator formula
\begin{align*}
    \bigl[ Y(a, z_1),\, Y(b, z_2) \bigr]
= &\Res_{z_0}\, z_2\, z_1^{-1}\, \delta\!\left( \frac{z_2(1 + z_0)}{z_1} \right)
\left( \frac{z_2(1 + z_0)}{z_1} \right)^{r/T}\\
&
\times Y\bigl( Y(a, \log(1 + z_0))\, b,\, z_2 \bigr).
\end{align*}
Then we obtain
\begin{align*}
&\operatorname{Res}_{z_0} z_0^iz_2^{i+1}\left((z_0+1)z_2\right)^{-1+ \delta_{0}(r) + r/T+j} Y\left(a, (z_0+1)z_2\right) Y\left(b, z_2\right) u  \\
=&\operatorname{Res}_{z_0}\Res_{z_1}z_1^{-1}\delta\left(\frac{(z_0+1)z_2}{z_1}\right) z_0^iz_2^{i+1}\left((z_0+1)z_2\right)^{-1+ \delta_{0}(r) + r/T+j} \\
&\quad\quad\quad\quad\quad\quad\quad\quad\quad\quad\quad\quad\quad\quad\quad \times Y\left(a, (z_0+1)z_2\right)Y\left(b, z_2\right) u  \\
=&\operatorname{Res}_{z_0}\Res_{z_1}(z_0z_2)^{-1}\delta\left(\frac{z_1-z_2}{z_0z_2}\right) (z_1-z_2)^iz_2z_1^{-1+ \delta_{0}(r) + r/T+j}  Y\left(a, z_1\right)Y\left(b, z_2\right) u  \\
= & \operatorname{Res}_{z_1}\left(z_1-z_2\right)^i z_1^{-1+ \delta_{0}(r) + r/T+j} Y\left(a, z_1\right) Y\left(b, z_2\right) u \\
= & \operatorname{Res}_{z_1}\left(z_1-z_2\right)^i z_1^{-1+ \delta_{0}(r) + r/T+j} Y\left(a, z_1\right) Y\left(b, z_2\right) u \\
& -\operatorname{Res}_{z_1}\left(z_1-z_2\right)^i z_1^{-1+ \delta_{0}(r) + r/T+j} Y\left(b, z_2\right) Y\left(a, z_1\right) u \\
= & \operatorname{Res}_{z_1}\left(z_1-z_2\right)^i z_1^{-1+ \delta_{0}(r) + r/T+j}\left[Y\left(a, z_1\right), Y\left(b, z_2\right)\right] u \\
= & \operatorname{Res}_{z_0} \operatorname{Res}_{z_1}\left(z_1-z_2\right)^i z_1^{-1+ \delta_{0}(r) + r/T+j}\\
&\quad z_2z_1^{-1} \delta\left(\frac{z_2(1+z_0)}{z_1}\right)\left(\frac{z_2(1+z_0)}{z_1}\right)^{r/T} Y\left(Y\left(u, \log(1+z_0)\right) v, z_2\right)  \\
= & \operatorname{Res}_{z_0} (z_2z_0)^iz_2\left((1+z_0)z_2\right)^{-1+ \delta_{0}(r) + r/T+j} Y\left(Y\left(a, \log(1+z_0)\right) b, z_2\right) u.
\end{align*}

\noindent\textbf{Claim 2.} For all $a \in V^r$, $b \in V$, $u' \in U^*$, $u \in U$ and $j \in \N$,
\begin{align*}
&\Res_{z_0}\, z_0^{-1} (z_0 + 1)^{-1+ \delta_{0}(r) + r/T + j}
   \left\langle u',\, Y\bigl(a, (z_0 + 1) z_2\bigr)\, Y(b, z_2)\, u \right\rangle \\
  = &\Res_{z_0}\, z_0^{-1} (1 + z_0)^{-1+ \delta_{0}(r) + r/T + j}
   \left\langle u',\, Y\bigl(Y(a, \log(1 + z_0))\, b, z_2\bigr)\, u \right\rangle.
\end{align*}
We now use Claim~1 to compute the left-hand side. Expanding $(z_0 + 1)^j$ via the binomial theorem yields
\begin{align*}
&\Res_{z_0}\, z_0^{-1} (z_0 + 1)^{-1+ \delta_{0}(r) + r/T + j}
   \left\langle u',\, Y\bigl(a, (z_0 + 1) z_2\bigr)\, Y(b, z_2)\, u \right\rangle \\
  =& \sum_{k=0}^{j} \binom{j}{k}
   \Res_{z_0}\, z_0^{k - 1} (z_0 + 1)^{-1+ \delta_{0}(r) + r/T}
   \left\langle u',\, Y\bigl(a, (z_0 + 1) z_2\bigr)\, Y(b, z_2)\, u \right\rangle \\
 =& \sum_{k=1}^{j} \binom{j}{k}
   \Res_{z_0}\, z_0^{k - 1} (1 + z_0)^{-1+ \delta_{0}(r) + r/T}
   \left\langle u',\, Y\bigl(Y(a, \log(1 + z_0))\, b, z_2\bigr)\, u \right\rangle \\
 & + \Res_{z_0}\, z_0^{-1} (z_0+1  )^{-1+ \delta_{0}(r) + r/T}
   \left\langle u',\, Y\bigl(a, (z_0 + 1) z_2\bigr)\, Y(b, z_2)\, u \right\rangle.
\end{align*}
Thus, it suffices to prove that
\begin{align*}
&\Res_{z_0}\, z_0^{-1} (z_0 + 1)^{-1+ \delta_{0}(r) + r/T}
   \left\langle u',\, Y\bigl(a, (z_0 + 1) z_2\bigr)\, Y(b, z_2)\, u \right\rangle \\
  = &\Res_{z_0}\, z_0^{-1} (1 + z_0)^{-1+ \delta_{0}(r) + r/T}
   \left\langle u',\, Y\bigl(Y(a, \log(1 + z_0))\, b, z_2\bigr)\, u \right\rangle.
\end{align*}
Without loss of generality, assume that $b \in V^s$ with $r + s \equiv 0 \pmod{T}$. Then we obtain
\begin{align}
&\Res_{z_0}\, z_0^{-1} (1 + z_0)^{-1+ \delta_{0}(r) + r/T}
   \left\langle u',\, Y\bigl(Y(a, \log(1 + z_0))\, b, z_2\bigr)\, u \right\rangle \nonumber\\
 = &\left\langle u',\, o\!\left(
      \Res_{z_0}\, \frac{(1 + z_0)^{-1+ \delta_{0}(r) + r/T}}{z_0}
      Y(a, \log(1 + z_0))\, b
   \right) u \right\rangle,\label{eqclaim2}
\end{align}
since $\langle u',\, \tilde{M}(U)(m) \rangle = 0$ for all $m \neq 0$. On the other hand, note that $o_p(b) u = 0$ whenever $p < 0$. Hence,
\begin{align}
&\Res_{z_0}\, z_0^{-1} (z_0 + 1)^{-1+ \delta_{0}(r) + r/T}
   \left\langle u',\, Y\bigl(a, (z_0 + 1) z_2\bigr)\, Y(b, z_2)\, u \right\rangle \nonumber \\
   = & \operatorname{Res}_{z_1}\left(z_1-z_2\right)^{-1} z_1^{-1+ \delta_{0}(r) + r/T }z_2^{-(-1+ \delta_{0}(r) + r/T )} \left\langle u',\, Y(a, z_1)\, Y(b, z_2)\, u \right\rangle \nonumber \\
  = &\left\langle u',\, 
   \sum_{i \in \N} o_{\,i - (-1+ \delta_{0}(r) + r/T)}(a)
   \sum_{{0\leq m \in r/T + \Z }} o_m(b)\,
   z_2^{-(-1+ \delta_{0}(r) + r/T) + i + m} u
   \right\rangle \nonumber \\
 =& \left\langle u',\, 
   \sum_{i \in \N} o_{\,i - (-1+ \delta_{0}(r) + r/T)}(a)\,
   o_{-i -1+ \delta_{0}(r) + r/T}(b)\, u
   \right\rangle. \label{eq6.1}
\end{align}
When $r=0$,
$
    \eqref{eq6.1}=\left\langle u^\prime,o(a)o(b)u\right\rangle
=\left\langle u^{\prime}, o\left(\operatorname{Res}_{z_0}{z_0^{-1}} Y(a, \log(1+z_0)) b \right) u\right\rangle=\eqref{eqclaim2}.
$
When $r>0$, $\eqref{eqclaim2}=\eqref{eq6.1}=0.$

  \noindent\textbf{Claim 3.} For all $a \in V^r$, $b \in V$, $u' \in U^*$, $u \in U$ and $j \in \N$,
\begin{align*}
&\left\langle u',\, (z_0 + 1)^{-1+ \delta_{0}(r) + r/T + j}\,
   Y\!\bigl(a, (z_0 + 1) z_2\bigr)\,
   Y(b, z_2)\, u \right\rangle \\
  = &\left\langle u',\, (1 + z_0)^{-1+ \delta_{0}(r) + r/T + j}\,
   Y\!\bigl(Y(a, \log(1 + z_0))\, b, z_2\bigr)\, u \right\rangle.
\end{align*}
To prove this identity, it suffices to show that for every $m \in \Z$,
\begin{align*}
&\Res_{z_0}\, z_0^{\,m} (z_0 + 1)^{-1+ \delta_{0}(r) + r/T + j}
   \left\langle u',\, Y\bigl(a, (z_0 + 1) z_2\bigr)\, Y(b, z_2)\, u \right\rangle \\
  = &\Res_{z_0}\, z_0^{\,m} (1 + z_0)^{-1+ \delta_{0}(r) + r/T + j}
   \left\langle u',\, Y\bigl(Y(a, \log(1 + z_0))\, b, z_2\bigr)\, u \right\rangle.
\end{align*}
This equality holds for all $m \geq -1$ by Claims~1 and~2.  
Now let $m = -k$ with $k \in \Z_{\geq1}$, and proceed by induction on $k$.
Assume the statement holds for all integers greater than $-k$.  
Applying the induction hypothesis to $\D a$ yields
\begin{align*}
&\Res_{z_0}\, z_0^{-k} (z_0 + 1)^{-1+ \delta_{0}(r) + r/T + j}
   \left\langle u',\, Y\bigl(\D a, (z_0 + 1) z_2\bigr)\, Y(b, z_2)\, u \right\rangle \\
 = &\Res_{z_0}\, z_0^{-k} (1 + z_0)^{-1+ \delta_{0}(r) + r/T + j}
   \left\langle u',\, Y\bigl(Y(\D a, \log(1 + z_0))\, b, z_2\bigr)\, u \right\rangle.
\end{align*}
Using the residue identity
$
\Res_z\, (f'(z) g(z) +  f(z) g'(z)) = 0,
$
together with the $\D$-derivation property \eqref{daozi} of vertex operators, then we have
\begin{align*}
&\operatorname{Res}_{z_0} z_0^{-k}\left(z_0+1\right)^{-1+ \delta_{0}(r) + r/T+j}\left\langle u^{\prime}, Y\left(\D a, (z_0+1)z_2\right) Y\left(b, z_2\right) u\right\rangle  \\
= & -\operatorname{Res}_{z_0}\left(\frac{\partial}{\partial z_0} z_0^{-k}\left(z_0+1\right)^{ \delta_{0}(r) + r/T+j}\right)\left\langle u^{\prime}, Y\left(a, (z_0+1)z_2\right) Y\left(b, z_2\right) u\right\rangle \\
= & \operatorname{Res}_{z_0} k z_0^{-k-1}\left(z_0+1\right)^{ \delta_{0}(r) + r/T+j}\left\langle u^{\prime}, Y\left(a, (z_0+1)z_2\right) Y\left(b, z_2\right) u\right\rangle \\
& -\operatorname{Res}_{z_0}( \delta_{0}(r) + r/T+j) z_0^{-k}\left(z_0+1\right)^{ \delta_{0}(r) + r/T-1+j} \\
& \times\left\langle u^{\prime}, Y\left(a, (z_0+1)z_2\right) Y\left(b, z_2\right) u\right\rangle \\
= & \operatorname{Res}_{z_0} k z_0^{-k-1} \left(z_0+1\right)^{-1+ \delta_{0}(r) + r/T+j}\left\langle u^{\prime}, Y\left(a,( z_0+1)z_2\right) Y\left(b, z_2\right) u\right\rangle \\
& +\operatorname{Res}_{z_0} k z_0^{-k}\left(z_0+1\right)^{-1+ \delta_{0}(r) + r/T+j}\left\langle u^{\prime}, Y\left(a, (z_0+1)z_2\right) Y\left(b, z_2\right) u\right\rangle \\
& -\operatorname{Res}_{z_0}( \delta_{0}(r) + r/T+j) z_0^{-k}\left(1+z_0\right)^{-1+ \delta_{0}(r) + r/T+j} \\
& \times\left\langle u^{\prime}, Y\left(Y\left(a, \log(1+z_0)\right) b, z_2\right) u\right\rangle \\
= & \operatorname{Res}_{z_0} k z_0^{-k-1} \left(z_0+1\right)^{-1+ \delta_{0}(r) + r/T+j}\left\langle u^{\prime}, Y\left(a, (z_0+1)z_2\right) Y\left(b, z_2\right) u\right\rangle \\
& +\operatorname{Res}_{z_0} k z_0^{-k}\left(1+z_0\right)^{-1+ \delta_{0}(r) + r/T+j}\left\langle u^{\prime}, Y\left(Y\left(a, \log(1+z_0)\right) b, z_2\right) u\right\rangle \\
& -\operatorname{Res}_{z_0}( \delta_{0}(r) + r/T+j) z_0^{-k}\left(1+z_0\right)^{-1+ \delta_{0}(r) + r/T+j} \\
& \times\left\langle u^{\prime}, Y\left(Y\left(a, \log(1+z_0)\right) b, z_2\right) u\right\rangle
\end{align*}
and
\begin{align*}
&\operatorname{Res}_{z_0}  z_0^{-k}\left(1+z_0\right)^{-1+ \delta_{0}(r) + r/T+j}\left\langle u^{\prime}, Y\left(Y\left(\D a, \log(1+z_0)\right) b, z_2\right) u\right\rangle \\
= & -\operatorname{Res}_{z_0}\left(\frac{\partial}{\partial z_0} z_0^{-k}\left(1+z_0\right)^{ \delta_{0}(r) + r/T+j}\right)\left\langle u^{\prime}, Y\left(Y\left(a, \log(1+z_0)\right) b, z_2\right) u\right\rangle \\
= & \left.\operatorname{Res}_{z_0} k z_0^{-k-1}\left(1+z_0\right)^{ \delta_{0}(r) + r/T+j}\left\langle u^{\prime}, Y\left(Y\left(a, \log(1+z_0)\right) b, z_2\right)\right\rangle u\right\rangle \\
& -\operatorname{Res}_{z_0}( \delta_{0}(r) + r/T+j) z_0^{-k}\left(1+z_0\right)^{-1+ \delta_{0}(r) + r/T+j} \\
& \times\left\langle u^{\prime}, Y\left(Y\left(a, \log(1+z_0)\right) b, z_2\right) u\right\rangle \\
= & \operatorname{Res}_{z_0} k  z_0^{-k-1}\left(1+z_0\right)^{-1+ \delta_{0}(r) + r/T+j}\left\langle u^{\prime}, Y\left(Y\left(a, \log(1+z_0)\right) b, z_2\right) u\right\rangle \\
& +\operatorname{Res}_{z_0} k z_0^{-k}\left(1+z_0\right)^{-1+ \delta_{0}(r) + r/T+j}\left\langle u^{\prime}, Y\left(Y\left(a, \log(1+z_0)\right) b, z_2\right) u\right\rangle \\
& -\operatorname{Res}_{z_0}( \delta_{0}(r) + r/T+j) z_0^{-k}\left(1+z_0\right)^{-1+ \delta_{0}(r) + r/T+j} \\
& \times\left\langle u^{\prime}, Y\left(Y\left(a, \log(1+z_0)\right) b, z_2\right) u\right\rangle .
\end{align*}
This yields the identity
\begin{align*}
& \operatorname{Res}_{z_0} z_0^{-k-1}\left(z_0+1\right)^{-1+ \delta_{0}(r) + r/T+j}\left\langle u^{\prime}, Y\left(a, (z_0+1)z_2\right) Y\left(b, z_2\right) u\right\rangle \\
 =&\operatorname{Res}_{z_0} z_0^{-k-1}\left(1+z_0\right)^{-1+ \delta_{0}(r) + r/T+j}\left\langle u^{\prime}, Y\left(Y\left(a, \log(1+z_0)\right) b, z_2\right) u\right\rangle.
\end{align*}

\noindent\textbf{Claim 4.} For all $a \in V^r$, $b \in V$, $u' \in U^*$ and $u \in \tilde{M}(U)$, there exists a positive integer $k$ such that
\begin{align}
&\left\langle u',\, (z_0 + 1)^{k + r/T}\,
   Y \!\bigl(a, (z_0 + 1) z_2\bigr)\,
   Y (b, z_2)\, u \right\rangle \nonumber \\
 =& \left\langle u',\, (1 + z_0)^{k + r/T}\,
   Y \!\bigl(Y(a, \log(1 + z_0))\, b, z_2\bigr)\, u \right\rangle.
\label{eq6.3}
\end{align}
Let $Z \subseteq \tilde{M}(U)$ denote the subspace consisting of all vectors $u$ such that  the identity~\eqref{eq6.3} holds.  
Our goal is to show that $Z = \tilde{M}(U)$.  
By Claim~3, we have $U \subseteq Z$. Since $\tilde{M}(U)$ is generated by $U$ as a $\hat{V}[g]$-module, it suffices to prove that $Z$ is stable under the action of $\hat{V}[g]$, i.e.,
$
\hat{V}[g]\, Z \subseteq Z.
$

To this end, let $u \in Z$, $c \in V^s$, and $0>m \in (1/T)\Z$.  
Choose a positive integer $k_1$ such that $c_i a = 0$ for all $i \geq k_1$.  
Since $u \in Z$, for each nonnegative integer $i$ there exists a positive integer $k_2$  such that
\begin{align*}
&\left\langle u',\, (z_0 + 1)^{k_2 + (r+s)/T}\,
   Y(c_i a, (z_0 + 1) z_2)\, Y(b, z_2)\, u \right\rangle \\
  = &\left\langle u',\, (1 + z_0)^{k_2 + (r+s)/T}\,
   Y\bigl(Y(c_i a, \log(1 + z_0))\, b, z_2\bigr)\, u \right\rangle, \\
&\left\langle u',\, (z_0 + 1)^{k_2 + r/T}\,
   Y(a, (z_0 + 1) z_2)\, Y(c_i b, z_2)\, u \right\rangle \\
  =& \left\langle u',\, (1 + z_0)^{k_2 + r/T}\,
   Y\bigl(Y(a, \log(1 + z_0))\, c_i b, z_2\bigr)\, u \right\rangle.
\end{align*}
We will also need the following elementary identity: for any $j \in \N$,
\begin{align}
   \sum_{i \geq j} \frac{m^i}{i!} \binom{i}{j} \bigl(\log(1 + z_0)\bigr)^{i - j}
= \frac{m^j}{j!} (1 + z_0)^m.
\label{eq6.4} 
\end{align}
Now choose a positive integer $k$ such that
$
k + {r}/{T} + m - k_1 > k_2 + {(r + s)}/{T}.
$
Using the identity~\eqref{eq6.4} together with the commutator formula (a consequence of Lemma~\ref{lem5.1}(1)):
\begin{align*}
  \bigl[ a(m),\, Y(b, z_2) \bigr]
= z_2^{\,m} \sum_{i=0}^{\infty} \frac{m^i}{i!}\, Y(a_i b, z_2), 
\end{align*}
we obtain
\begin{align*}
& \left\langle u^{\prime},\left(z_0+1\right)^{k+\frac{r}{T}} Y\left(a, (z_0+1)z_2\right) Y\left(b, z_2\right) c(m) u\right\rangle \\
= & \left\langle u^{\prime},\left(z_0+1\right)^{k+\frac{r}{T}} c(m) Y\left(a, (z_0+1)z_2\right) Y\left(b, z_2\right) u\right\rangle \\
& -\left\langle u^{\prime},\sum_{i=0}^{\infty}\frac{m^i}{i!}\left(z_0+1\right)^{k+\frac{r}{T}+m}z_2^m Y\left(c_i a, (z_0+1)z_2\right) Y\left(b, z_2\right) u\right\rangle \\
& -\left\langle u^{\prime},\sum_{i=0}^{\infty}\frac{m^i}{i!} z_2^{m}\left(z_0+1\right)^{k+\frac{r}{T}} Y\left(a,( z_0+1)z_2\right) Y\left(c_i b, z_2\right) u \right\rangle\\
=& -\left\langle u^{\prime},\sum_{i=0}^{\infty}\frac{m^i}{i!}\left(1+z_0\right)^{k+\frac{r}{T}+m}z_2^m Y\left(Y\left(c_i a, \log(1+z_0)\right) b, z_2\right) u \right\rangle\\
& -\left\langle u^{\prime},\sum_{i=0}^{\infty}\frac{m^i}{i!} z_2^{m}\left(1+z_0\right)^{k+\frac{r}{T}} Y\left(Y\left(a,\log( 1+z_0)\right) c_i b, z_2\right) u\right\rangle \\
=& -\left\langle u^{\prime},\sum_{i=0}^{\infty}\frac{m^i}{i!}\left(1+z_0\right)^{k+\frac{r}{T}+m}z_2^m Y\left(Y\left(c_i a, \log(1+z_0)\right) b, z_2\right) u \right\rangle\\
& -\left\langle u^{\prime},\sum_{i=0}^{\infty}\frac{m^i}{i!} z_2^{m}\left(1+z_0\right)^{k+\frac{r}{T}} Y\left(c_iY\left(a,\log( 1+z_0)\right)  b, z_2\right) u \right\rangle\\
& +\left\langle u^{\prime},\sum_{i=0}^{\infty} \sum_{j=0}^{\infty}\frac{m^i}{i!}\binom{i}{j} z_2^{m}\left(1+z_0\right)^{k+\frac{r}{T}} {\log(1+z_0)}^{i-j} Y\left(Y\left(c_j a, \log(1+z_0)\right) b, z_2\right) u \right\rangle\\
=& -\left\langle u^{\prime},\sum_{i=0}^{\infty}\frac{m^i}{i!}\left(1+z_0\right)^{k+\frac{r}{T}+m}z_2^m Y\left(Y\left(c_i a, \log(1+z_0)\right) b, z_2\right) u \right\rangle\\
& -\left\langle u^{\prime},\sum_{i=0}^{\infty}\frac{m^i}{i!} z_2^{m}\left(1+z_0\right)^{k+\frac{r}{T}} Y\left(c_iY\left(a,\log( 1+z_0)\right)  b, z_2\right) u \right\rangle\\
& +\left\langle u^{\prime},\sum_{j=0}^{\infty} \sum_{i\geq j}\frac{m^i}{i!}\binom{i}{j} z_2^{m}\left(1+z_0\right)^{k+\frac{r}{T}} {\log(1+z_0)}^{i-j} Y\left(Y\left(c_j a, \log(1+z_0)\right) b, z_2\right) u \right\rangle\\
=& -\left\langle u^{\prime},\sum_{i=0}^{\infty}\frac{m^i}{i!}\left(1+z_0\right)^{k+\frac{r}{T}+m}z_2^m Y\left(Y\left(c_i a, \log(1+z_0)\right) b, z_2\right) u \right\rangle\\
& -\left\langle u^{\prime},\sum_{i=0}^{\infty}\frac{m^i}{i!} z_2^{m}\left(1+z_0\right)^{k+\frac{r}{T}} Y\left(c_iY\left(a,\log( 1+z_0)\right)  b, z_2\right) u\right\rangle \\
& +\left\langle u^{\prime},\sum_{j=0}^{\infty}\frac{m^j}{j!} z_2^{m}\left(1+z_0\right)^{k+\frac{r}{T}+m}  Y\left(Y\left(c_j a, \log(1+z_0)\right) b, z_2\right) u \right\rangle\\
=& -\left\langle u^{\prime},\sum_{i=0}^{\infty}\frac{m^i}{i!} z_2^{m}\left(1+z_0\right)^{k+\frac{r}{T}} Y\left(c_iY\left(a,\log( 1+z_0)\right)  b, z_2\right) u\right\rangle \\
=& -\left\langle u^{\prime},\left(1+z_0\right)^{k+\frac{r}{T}} [c(m),Y\left(Y\left(a,\log( 1+z_0)\right)  b, z_2\right)] u\right\rangle \\
= & \left\langle u^{\prime},\left(1+z_0\right)^{k+\frac{r}{T}}  Y\left(Y( a, \log(1+z_0)\right) b, z_2)c(m) u\right\rangle.
\end{align*}
From the proof of Claim 4, we see that \eqref{eq6.5} holds with $X d(i)$ ($d\in V,i\in(1/T)\Z$) in place of $X$. It then follows from induction.
\end{proof}

Set
\[
J (U) = \left\{\, v \in \tilde{M} (U) \;\middle|\;
\langle u',\, x v \rangle = 0 \text{ for all } u' \in U^* \text{ and all } x \in U(\hat{V}[g]) \,\right\}.
\]
Then $J (U)$ is a graded $U(\hat{V}[g])$-submodule of $\tilde{M} (U)$.
Let $W \subseteq \tilde{M} (U)$ be the subspace spanned linearly by the coefficients (in $z_0$ and $z_2$) of the formal series
\begin{align}
& (z_0 + 1)^{-1+ \delta_{0}(r) + r/T}\,
   Y\bigl(a, (z_0 + 1) z_2\bigr)\, Y(b, z_2)\, u \nonumber \\
&\quad - (1 + z_0)^{-1+ \delta_{0}(r) + r/T}\,
   Y\bigl(Y(a, \log(1 + z_0))\, b, z_2\bigr)\, u,
\label{eq6.6}
\end{align}
for all $a \in V^r$, $b \in V$, and $u \in U$.
Define the quotient module
\[
\bar{M} (U) = \tilde{M} (U) \big/ U(\hat{V}[g])\, W.
\]
The first main result of this section is the following:
\begin{thm}\label{thm6.3}
    
The space $\bar{M} ({U})=\bigoplus_{m \in(1/T)\N} \bar{M} ({U})({m})$ is a $(1/T)\mathbb{N}$-graded $g$-twisted $\phi$-coordinated $V$-module with $\bar{M}(U)(0)=U$ and with the following universal property: for any $g$-twisted $\phi$-coordinated $V$-module $M$ and any $\tilde{A}_{g}(V)$-morphism $\varphi: U \rightarrow \tilde{\Omega}_0(M)$, there is a unique morphism $\bar{\varphi}: \bar{M}({U}) \rightarrow {M}$ of $g$-twisted $\phi$-coordinated $V$-modules which extends $\varphi$.
\end{thm}
\begin{proof}

By definition, \eqref{eq6.6} holds for $a, b \in V$ and for $u \in U + U(\hat{V}[g]) W \subset \bar{M}(U)$. It follows from the proof of Claim~4 in the proof of Lemma~\ref{lem6.1} that \eqref{eq6.6} in fact holds for all $a, b \in V$ and for every $u \in \bar{M}(U)$. Consequently, $\bar{M}(U)$ is an $(1/T)\mathbb{N}$-graded $g$-twisted $\phi$-coordinated $V$-module.
By Proposition~\ref{prop2.8}, $\bar{M}(U)$ is spanned by
$$
\left\{ a(n)\bigl(U + U(\hat{V}[g]) W\bigr) \mid a \in V,\ n \in (1/T)\mathbb{Z} \right\}.
$$
Hence, for all $m \in (1/T)\mathbb{N}$, we have
$$
\bar{M}(U)(m) = \hat{V}[g]_m \bigl(U + U(\hat{V}[g]) W\bigr).
$$
In particular, 
$
\bar{M}(U)(0) = \tilde{A}_{g}(V)\bigl(U + U(\hat{V}[g]) W\bigr),
$
which is a quotient module of $U$. From Lemma~\ref{lem6.1}, we know that $U(\hat{V}[g])W \subset J(U)$ and that $J(U) \cap U = 0$. Therefore, $\bar{M}(U)(0)$ contains $U$ as a submodule, which implies $\bar{M}(U)(0) = U$.
Finally, the universal property of $\bar{M}(U)$ follows directly from its construction.
\end{proof}

Furthermore, we set
$$
\tilde{L} (U)=\tilde{M} (U) / J (U)
$$
which is a $(1/T)\mathbb{N}$-graded ${U}(\hat{V}[g])$-module.
We can now state the second 
main result of this section.

\begin{thm}\label{thm6.2}
$\tilde{L} (U)$ is a $(1/T)\mathbb{N}$-graded $g$-twisted $\phi$-coordinated $V$-module with $U$ as the degree $0$ subspace. Furthermore, if $U$ is an irreducible $\tilde{A}_{g}(V)$-module, then $\tilde{L} (U)$ is an irreducible $(1/T)\mathbb{N}$-graded $g$-twisted $\phi$-coordinated $V$-module.
\end{thm}

\begin{proof}

By Lemma~\ref{lem6.1}, we know that $\tilde{L} (U)$ is a $(1/T)\mathbb{N}$-graded $g$-twisted $\phi$-coordinated $V$-module and that $J (U) \cap U = 0$.
From the construction, $\tilde{L} (U)$, regarded as a $U(\hat{V}[g])$-module, is a homomorphic image of $\bar{M} (U)$. Consequently, we have $\tilde{L} (U)(0) = U$.
Moreover, it is straightforward to verify that if $U$ is an irreducible $\tilde{A}_{g }(V)$-module, then $\tilde{L} (U)$ is an irreducible $(1/T)\mathbb{N}$-graded $g$-twisted $\phi$-coordinated $V$-module.
\end{proof}

From Proposition \ref{prop4.5} and Theorem \ref{thm6.2}, we have:

\begin{thm}\label{thm6.4}
    There exists a bijection between the simple $\tilde{A}_{g }(V)$-modules   and the irreducible $(1/T)\mathbb{N}$-graded $g$-twisted $\phi$-coordinated $V$-modules.
\end{thm}

\section{Universal enveloping algebra and isomorphisms}

In this section, we construct the universal enveloping algebra $U(V[g])$ and show that every $g$-twisted $\phi$-coordinated $V$-module naturally carries the structure of a $U(V[g])$-module. Moreover, we establish a precise relationship between   $\tilde{A}_{g }(V)$ and   $U(V[g])$.

Let $U(\hat{V}[g])$ denote the universal enveloping algebra of the Lie algebra $\hat{V}[g]$. The $(1/T)\Z$-grading on $\hat{V}[g]$ induces a $(1/T)\Z$-grading on $U(\hat{V}[g])$, namely
\[
U(\hat{V}[g]) = \bigoplus_{m \in (1/T)\Z} U(\hat{V}[g])_m.
\]
For $0>k \in -(1/T)\N$, define
\[
U(\hat{V}[g])_m^{\,k} := \sum_{k\geq i \in (1/T)\Z} U(\hat{V}[g])_{m - i}\, U(\hat{V}[g])_i,
\]
and set $U(\hat{V}[g])_m^{\,0} := U(\hat{V}[g])_m$. Then
\[
U(\hat{V}[g])_m^{\,k} \subseteq U(\hat{V}[g])_m^{\,k + 1/T},
\]
and
\[
\bigcap_{k \in -(1/T)\N} U(\hat{V}[g])_m^{\,k} = 0,
\qquad
\bigcup_{k \in -(1/T)\N} U(\hat{V}[g])_m^{\,k} = U(\hat{V}[g])_m.
\]
Thus, the family $\{ U(\hat{V}[g])_m^{\,k} \mid k \in -(1/T)\N \}$ forms a fundamental system of neighborhoods of $0$ in $U(\hat{V}[g])_m$. Let $\widetilde{U}(\hat{V}[g])_m$ be the completion of $U(\hat{V}[g])_m$ with respect to this topology, and define
\[
\widetilde{U}(\hat{V}[g]) := \bigoplus_{m \in (1/T)\Z} \widetilde{U}(\hat{V}[g])_m.
\]
Then $\widetilde{U}(\hat{V}[g])$ is a complete topological ring that admits infinite sums of homogeneous elements.

 We require the component form of the $g$-twisted $\phi$-Jacobi identity.  
For $u \in V^r$, $v \in V^s$, $l \in \mathbb{Z}$, $m \in {r}/{T} + \mathbb{Z}$, and $n \in {s}/{T} + \mathbb{Z}$, applying the operator
$
\operatorname{Res}_z \operatorname{Res}_{x_1} \operatorname{Res}_{x_2}\, (z x_2)^l x_1^{\,m} x_2^{\,n+1}
$
to the $g$-twisted $\phi$-Jacobi identity~\eqref{phi_jacobi} yields
\begin{align*}
&\sum_{i \geq 0} (-1)^i \binom{l}{i} \Bigl( u_{\,l + m - i}\, v_{\,n + i} - (-1)^l\, v_{\,l + n - i}\, u_{\,m + i} \Bigr) \\
&\qquad = \sum_{i,j \geq 0} (-1)^i \binom{l}{i} \frac{(l + m - i + 1)^j}{j!} \,(u_j v)_{\,l + m + n + 1}.
\end{align*}
Using this identity, we can formulate the definition of the universal enveloping algebra:

\begin{defi}
The \emph{universal enveloping algebra} $U(V[g])$ of $V$ with respect to $g$ is defined as the quotient of $\widetilde{U}(\hat{V}[g])$ by the two-sided ideal generated by the following relations:

{\rm(1)} $\mathbf{1}(0) = 1$ and $\mathbf{1}(i) = 0$ for all $i \in \Z \setminus \{0\}$;
    
{\rm(2)} For all $u \in V^r$, $v \in V^{r'}$, $l \in \Z$, $s \in r/T + \Z$, and $t \in r'/T + \Z$,
    \begin{align*}
    &\sum_{i \geq 0} (-1)^i \binom{l}{i} \Bigl( u(s - i) v(t + i) - (-1)^l v(l + t - i) u(s + i - l) \Bigr) \\
    &\qquad = \sum_{i,j \geq 0} (-1)^i \binom{l}{i} \frac{(s - i)^j}{j!} (u_j v)(s + t).
    \end{align*}
 
\end{defi}

It follows immediately that $U(V[g]) = \bigoplus_{m \in (1/T)\Z} U(V[g])_m$ is a $(1/T)\Z$-graded associative algebra. For $k \in -(1/T)\N$, set
\[
U(V[g])_m^{\,k} := \sum_{{k\geq i \in (1/T)\Z}} U(V[g])_{m - i}\, U(V[g])_i.
\]
Then $U(V[g])_0 /U(V[g])_0^{-1/T}$ is an associative algebra.

\begin{rmk} \label{rmk8.2}
{\rm(1)} By construction, any $g$-twisted $\phi$-coordinated $V$-module $(W, Y_W)$ becomes a $U(V[g])$-module via the action
    \[
    u(m) \cdot w := u_m w, \quad \text{for } u \in V^r,\ m \in r/T + \Z,\ w \in W.
    \]

 {\rm(2)} We shall continue to denote by $u(s)$ the image of the element $u(s) \in \widetilde{U}(\hat{V}[g])$ in $U(V[g])$ or its quotients.

\end{rmk}

\begin{lem}\label{lem8.3}
Let $n, m \in (1/T)\N$. For any element
\[
w = \sum u_1(m_1) \cdots u_k(m_k) \in U(V[g])_{n - m} \big/ U(V[g])_{n - m}^{-m - 1/T},
\]
where $u_j \in V$ and $m_j \in (1/T)\Z$, there exists $u_{k+1} \in V$ such that
$
w = u_{k+1}(m - n).
$
\end{lem}

\begin{proof}
We prove by induction on the lexicographical order of pairs $(k, -m_k)$ that for any monomial $u_1(m_1) \cdots u_k(m_k)$, there exists $u_{k+1} \in V$ such that
\[
u_1(m_1) \cdots u_k(m_k) \equiv u_{k+1}(m - n) \mod U(V[g])_{n - m}^{-m - 1/T}.
\]
The claim is clear when $k = 1$ or when $m_k > m$. Suppose, for contradiction, that the statement fails for some monomial, and let $(l, -m_l)$ be a minimal counterexample with respect to lexicographical order. Write
$
s = m_{l-1}$,  $t = m_l$,  $u = u_{l-1}$,  $v = u_l,
$
and denote by $J(l-2)$ the product $u_1(m_1) \cdots u_{l-2}(m_{l-2})$. Using the defining relation of $U(V[g])$, we compute
\begin{align*}
&\,J(l-2) u(s) v(t)\\
=& -\sum_{k \geq 1} (-1)^k \binom{s - \xqz{m} - 1}{k} J(l-2) u(s - k) v(t + k) \\
& + \sum_{k \geq 0} (-1)^{k + s - \xqz{m} - 1} \binom{s - \xqz{m} - 1}{k} J(l-2) v(s + t - \xqz{m} - k - 1) u(\xqz{m} + k + 1) \\
& + \sum_{k,j \geq 0} (-1)^k \binom{s - \xqz{m} - 1}{k} \frac{(s - k)^j}{j!} J(l-2) (u_j v)(s + t).
\end{align*}
Modulo $U(V[g])_{n - m}^{-m - 1/T}$, the second term vanishes because $\xqz{m} + k + 1 > m$ for all $k \geq 0$. Thus,
\begin{align*}
&\,J(l-2) u(s) v(t)\\
\equiv& -\sum_{k \geq 1} (-1)^k \binom{s - \xqz{m} - 1}{k} J(l-2) u(s - k) v(t + k) \\
& + \sum_{k,j \geq 0} (-1)^k \binom{s - \xqz{m} - 1}{k} \frac{(s - k)^j}{j!} J(l-2) (u_j v)(s + t)
\mod U(V[g])_{n - m}^{-m - 1/T}.
\end{align*}
Each term on the right-hand side corresponds to a pair strictly smaller than $(l, -m_l)$ in lexicographical order: the first sum involves $(l, -m_l - k)$ with $k \geq 1$, and the second involves $(l - 1, -m_l - m_{l-1})$. This contradicts the minimality of $(l, -m_l)$, completing the proof.
\end{proof}

Analogously to Lemma~\ref{lem4.1}, we obtain the following result.

\begin{lem}\label{lem8.4}
Let $u, v \in V$ and $m, n, p \in (1/T)\N$. Then
\[
\bigl(u \bullet_{g,m,\,p}^{\,n} v\bigr)(m - n)
\equiv u(p - n) v(m - p)
\mod U(V[g])_{n - m}^{-m - 1/T}.
\]
In particular, $(u\bullet_gv)(0)\equiv u(0)v(0)\bmod U(V[g])_0^{-1/T}.$
\end{lem}

Set $\mathbb{M} =\bigoplus_{n\in (1/T)\N}\mathbb M (n)$ with  $\mathbb{M} (n)=U(V[g])_{n } / U(V[g])_{n }^{ -1/T}$ for $n \in (1/T)\N$. 
We equip $\mathbb{M} $ with the vertex operator maps $Y_{\mathbb{M} }(u, z)=\sum_{p \in r/T+\mathbb{Z}} u_{p} z^{-p-1}$ for  $u\in V^r$, where for $n \in (1/T)\N$, the linear map $u_p$ from $\mathbb{M} (n)$ to $\mathbb{M} (n  -p-1)$ is defined as follows:
$$u_{p}(v  )=\left\{\begin{array}{ll}
u(p+1) v  , & \text { if } n -p-1 \geq 0, \\
0, & \text { if } n-p-1<0,
\end{array}\right.$$
for $v \in U(V[g])_{n } / U(V[g])_{n }^{ -1/T}$.  Then $\mathbb{M}$ is a $(1/T)\N$-graded $g$-twisted $\phi$-coordinated $V$-module, since the $g$-twisted $\phi$-Jacobi identity follows immediately from the construction of $U(V[g])$; and $\mathbb{M}$ is generated by $\mathbb{M}(0)$.

\begin{thm}\label{thm:iso}
Define a linear map
\[
\varphi  : \tilde{A}_{g }(V) \longrightarrow
U(V[g])_{0} \big/ U(V[g])_{0}^{ - 1/T}
\]
by
$
\varphi \bigl(u + \tilde{O}_{g }(V)\bigr)
:= u(0) + U(V[g])_{0}^{ - 1/T}.
$
Then  $\varphi$ is an algebra isomorphism.

\end{thm}

\begin{proof}
Since $\mathbb{M}(0) = U(V[g])_0 / U(V[g])_0^{ - 1/T} \subseteq \tilde{\Omega}_0(\mathbb{M})$, for any $u \in \tilde{O}_{g}(V)$, it follows from Lemma~\ref{lem4.3} that
$
0 = o(u)(\mathbf{1}(0) + U(V[g])_0^{  - 1/T}) = u(0) + U(V[g])_{0}^{ - 1/T}.
$
Hence, $\varphi$ is well-defined.
Surjectivity of $\varphi $ follows from Lemma~\ref{lem8.3}. For injectivity, suppose $u \in V$ satisfies
$
u(0) \in U(V[g])_{0}^{ - 1/T}.
$
Then, by Remark~\ref{rmk8.2}(1), the operator $o(u)$ acts as zero on $\tilde{\Omega}_0(M)$ for every $g$-twisted $\phi$-coordinated $V$-module $M$. Consider the $(1/T)\N$-graded $g$-twisted $\phi$-coordinated $V$-module $\bar{M} (\tilde{A}_{g }(V))$ from Theorem \ref{thm6.3}, so $\bar{M} (\tilde{A}_{g }(V))(0) = \tilde{A}_{g }(V)\subseteq \tilde{\Omega}_0(\bar{M} (A_{g}(V)))$ by Proposition~\ref{prop4.5}\,(1), then we have
$$
0 = o(u)(\mathbf{1} + \tilde{O}_{g }(V)) = u \bullet_{g } \mathbf{1} + \tilde{O}_{g }(V) = u + \tilde{O}_{g }(V),
$$
which implies  $u \in \tilde{O}_{g }(V)$. Hence   $\varphi$ is injective.
Let $u, v \in V$, we compute
\begin{align*}
&\varphi \!\left( \bigl(u + \tilde{O}_{g }(V)\bigr) \bullet_{g }  \bigl(v + \tilde{O}_{g }(V)\bigr) \right)   \\=& \varphi \!\left( u \bullet_{g}  v + \tilde{O}_{g }(V) \right)  = \bigl(u \bullet_{g }  v\bigr)(0) + U(V[g])_{0}^{ - 1/T} \\
 =& u(0) v(0) + U(V[g])_{0}^{  - 1/T}  = \bigl(u(0) + U(V[g])_{0}^{  - 1/T}\bigr) \cdot \bigl(v(0) + U(V[g])_0^{ - 1/T}\bigr)  \\ =& \varphi \!\left(u + \tilde{O}_{g }(V)\right) \cdot \varphi \!\left(v + \tilde{O}_{g }(V)\right),
\end{align*}
 where the third equality follows from Lemma \ref{lem8.4}, which implies $\varphi $ is an algebra homomorphism.
This completes the proof.
\end{proof}

\section{When $V$ is a vertex operator algebra}

In this section, assuming that $V$ is a vertex operator algebra, we show that each $\tilde{A}_{g,n,m}(V)$ is isomorphic to the $A_{g,n}(V)$--$A_{g,m}(V)$-bimodule $A_{g,n,m}(V)$ constructed by Dong and Jiang~\cite{DJ2}. Furthermore, we establish a bijection between irreducible admissible $g$-twisted $V$-modules and irreducible $(1/T)\mathbb{N}$-graded $g$-twisted $\phi$-coordinated $V$-modules.

\begin{defi}
A vertex operator algebra is a quadruple $(V, Y, \mathbf{1}, \omega)$, where $V=\bigoplus_{n \in \mathbb{Z}} V_n$ is a $\mathbb{Z}$-graded vector space with $\operatorname{dim} V_n<\infty$ for all $n \in \mathbb{Z}, V_n=0$ for $n \ll 0, \mathbf{1} \in V_0$, $\omega \in V_2$ and $(V, Y, \mathbf{1})$ is a vertex algebra, such that 
the Virasoro algebra relations hold: $[L(m), L(n)]=(m-n) L(m+n)+\frac{m^3-m}{12} c_V$ for $m, n \in \mathbb{Z}$, where $c_V \in \mathbb{C}$ and $L(m)=\omega_{m+1}$ for $m \in \mathbb{Z} ;\left.L(0)\right|_{V_m}=m \mathrm{id}_{V_m}$ for $m \in \mathbb{Z}$ and $Y(L(-1) w, z)=\frac{d}{d z} Y(w, z)$ for $w \in V$
\end{defi}
For any $n \in \mathbb{Z}$, elements of $V_n$ are called \emph{homogeneous}, and for homogeneous $u \in V_n$, we define $\operatorname{wt} u := n$. Whenever $\operatorname{wt} u$ appears, $u$ is assumed to be homogeneous.

\begin{defi}
Let $(V, Y, \mathbf{1}, \omega)$ be a vertex operator algebra. A linear automorphism $g$ of $V$ is called an \emph{automorphism of $V$} if $g$ is a vertex algebra automorphism of $(V, Y, \mathbf{1})$ and satisfies $g(\omega) = \omega$.
\end{defi}

Fix an automorphism $g$ of $V$ of order $T$. Then $V$ decomposes into eigenspaces for $g$:
\[
V = \bigoplus_{r=0}^{T-1} V^r, \qquad 
V^r = \left\{ v \in V \mid g v = e^{-2\pi \sqrt{-1} r / T} v \right\}.
\]

\begin{defi}
A \emph{weak $g$-twisted $V$-module} is a vector space $M$ equipped with a linear map
\[
Y_M(\cdot, z): V \to (\operatorname{End} M)[[z^{1/T}, z^{-1/T}]],
\quad
u \mapsto Y_M(u, z) = \sum_{n \in r/T + \mathbb{Z}} u_n z^{-n-1}
\quad (u \in V^r),
\]
satisfying the following axioms:

{\rm(1)} $Y_M(\mathbf{1}, z) = \mathrm{id}_M$;
    
{\rm(2)} For $u \in V^r$ and $w \in M$, $u_{r/T + n} w = 0$ for all sufficiently large $n$;
    
{\rm(3)}  For $u \in V^r$, $v \in V^s$,
    \[
    \begin{aligned}
    &z_0^{-1} \delta\!\left( \frac{z_1 - z_2}{z_0} \right) Y_M(u, z_1) Y_M(v, z_2)
    - z_0^{-1} \delta\!\left( \frac{z_2 - z_1}{-z_0} \right) Y_M(v, z_2) Y_M(u, z_1) \\
    &\qquad = z_2^{-1} \left( \frac{z_1 - z_0}{z_2} \right)^{-r/T}
    \delta\!\left( \frac{z_1 - z_0}{z_2} \right) Y_M(Y(u, z_0)v, z_2).
    \end{aligned}
    \]

\end{defi}

\begin{defi}
An \emph{admissible $g$-twisted $V$-module} is a   $(1/T)\mathbb{N}$-graded weak $g$-twisted $V$-module $M = \bigoplus_{n \in (1/T)\mathbb{N}} M(n)$, such that for all  $v \in V$ and $m, n \in (1/T)\mathbb{Z}$,
\[
v_m M(n) \subseteq M\bigl(n + \operatorname{wt} v - m - 1\bigr).
\]
\end{defi}

For $u \in V^r$, $v \in V$, and $m, n, p \in (1/T)\mathbb{N}$, define a product $*_{g,m,p}^{\,n}$ on $V$ as follows:
\begin{align*}
u *_{g,m,p}^{\,n} v
&= \sum_{i=0}^{\lfloor p \rfloor} (-1)^i 
\binom{\lfloor m \rfloor + \lfloor n \rfloor - \lfloor p \rfloor - 1 + \delta_{\bar{m}}(r) + \delta_{\bar{n}}(T - r) + i}{i} \\
&\quad \cdot \Res_z \frac{(1+z)^{\operatorname{wt} u - 1 + \lfloor m \rfloor + \delta_{\bar{m}}(r) + r/T}}{z^{\lfloor m \rfloor + \lfloor n \rfloor - \lfloor p \rfloor + \delta_{\bar{m}}(r) + \delta_{\bar{n}}(T - r) + i}} Y(u, z) v,
\end{align*}
if $\bar{p} - \bar{n} \equiv r \pmod{T}$; otherwise, set $u *_{g,m,p}^{\,n} v = 0$.
Set
$
\bar{*}_{g,m}^{\,n} := *_{g,m,n}^{\,n}$ and
$*_{g,m}^{\,n} := *_{g,m,m}^{\,n}.
$
Note that when $g = 1$, the product $*_{1,m,p}^{\,n}$ coincides with $*_{m,p}^{\,n}$ defined in~\cite{DJ1}. Moreover, $*_{g,n,n}^{\,n}$ is precisely the product $*_{g,n}$ introduced in~\cite{DLM3}; in particular, $*_{g,0}$ is the product $*_g$ from~\cite{DLM2}.

For $m, n \in (1/T)\mathbb{N}$, define
\[
O'_{g,n,m}(V) := \operatorname{span}\{ u \circ_{g,m}^{\,n} v \mid u, v \in V \} + L_{n,m}(V),
\]
where
\[
L_{n,m}(V) := \operatorname{span}\{ (L(-1) + L(0) + m - n) u \mid u \in V \},
\]
and for $u \in V^r$, $v \in V$,
\[
u \circ_{g,m}^{\,n} v
:= \Res_z \frac{(1+z)^{\operatorname{wt} u - 1 + \lfloor m \rfloor + \delta_{\bar{m}}(r) + r/T}}{z^{\lfloor m \rfloor + \lfloor n \rfloor + \delta_{\bar{m}}(r) + \delta_{\bar{n}}(T - r) + 1}} Y(u, z) v.
\]
When $m = n$, this reduces to the operation $u \circ_{g,n} v$ defined in~\cite{DLM3}. Set
$
O_{g,n}(V) :=O'_{g,n,n}(V) $, $ O_g(V) := O_{g,0}(V),
$
and  
\[
A_{g,n}(V) := V / O_{g,n}(V), \quad A_g(V) := A_{g,0}(V).
\]
The following result is due to Dong, Li, and Mason~\cite[Theorems~2.4 and~4.3]{DLM3}.

\begin{thm} \label{thm9.5}
{\rm(1)} The product $*_{g,n}$ induces an associative algebra structure on $A_{g,n}(V)$ with identity $\mathbf{1} + O_{g,n}(V)$.
    
{\rm(2)} There is a bijection between irreducible $A_{g,n}(V)$-modules that do not factor through $A_{g,n - 1/T}(V)$ and irreducible admissible $g$-twisted $V$-modules. In particular, there is a bijection between irreducible $A_{g}(V)$-modules  and irreducible admissible $g$-twisted $V$-modules.

\end{thm}

Let $O''_{g,n,m}(V)$ be the linear span of all elements of the form
\[
u *_{g,m,p_3}^{\,n} \Bigl( (a *_{g,p_1,p_2}^{\,p_3} b) *_{g,m,p_1}^{\,p_3} c - a *_{g,m,p_2}^{\,p_3} (b *_{g,m,p_1}^{\,p_2} c) \Bigr),
\]
for $a,b,c,u \in V$ and $p_1, p_2, p_3 \in (1/T)\mathbb{N}$. Define
\[
O'''_{g,n,m}(V) := \sum_{p_1, p_2 \in (1/T)\mathbb{N}} \left( V *_{g,p_1,p_2}^{\,n} O'_{g,p_2,p_1}(V) \right) *_{g,m,p_1}^{\,n} V,
\]
and set
\[
O_{g,n,m}(V) := O'_{g,n,m}(V) + O''_{g,n,m}(V) + O'''_{g,n,m}(V), \qquad
A_{g,n,m}(V) := V / O_{g,n,m}(V).
\]

\begin{thm}[\cite{DJ2}] \label{T2.3}
Let $V$ be a vertex operator algebra and $m, n \in (1/T)\mathbb{N}$. Then $A_{g,n,m}(V)$ is an $A_{g,n}(V)$--$A_{g,m}(V)$-bimodule, where the left and right actions are induced by $\bar{*}_{g,m}^{\,n}$ and $*_{g,m}^{\,n}$, respectively.
\end{thm}

For a vertex operator algebra $(V, \omega)$ of central charge $c$, there is a second vertex operator algebra structure on the same underlying vector space, denoted $\exp(V, \omega)$, with vertex operators defined by
\[
Y[u, z] := Y\bigl(e^{z L(0)} u, e^z - 1\bigr), \quad u \in V,
\]
vacuum vector $\mathbf{1}$, and new conformal vector $\tilde{\omega} = \omega - \frac{c}{24} \mathbf{1}$ (see~\cite{Z1}). Denote the corresponding Virasoro operators by
$
Y[\tilde{\omega}, z] = \sum_{n \in \mathbb{Z}} L[n] z^{-n-2}.
$

\begin{lem} \label{lem9.8}
Let $(V, \omega)$ be a vertex operator algebra, $g$ an automorphism of order $T$, and $n, m \in (1/T)\mathbb{N}$. Then
\[
A_{g,n}(V, \omega) = \tilde{A}_{g,n}\bigl(\exp(V, \omega)\bigr),
\qquad
A_{g,n,m}(V, \omega) = \tilde{A}_{g,n,m}\bigl(\exp(V, \omega)\bigr).
\]
\end{lem}

\begin{proof}
Since $g(\omega) = \omega$, it follows that $g$ also preserves $\tilde{\omega}$, hence $g$ is an automorphism of $\exp(V, \omega)$.
For $u \in V^r$, $v \in V$, and $m_1, n_1, p \in (1/T)\mathbb{N}$ with $\bar{p} - \bar{n}_1 \equiv r \pmod{T}$, we compute:
\begin{align*}
    u *_{g, m_1, p}^{n_1} v=&\sum_{i=0}^{\lfloor p\rfloor}(-1)^{i}
    \binom{\lfloor m_1\rfloor+\lfloor n_1\rfloor-\lfloor p\rfloor-1+\delta_{\bar{m}_1}(r)+\delta_{\bar{n}_1}(T-r)+i}{i} \\
&\cdot \operatorname{Res}_{z} \frac{(1+z)^{-1+\lfloor m_1\rfloor+\delta_{\bar{m}_1}(r)+r / T}}{z^{\lfloor m_1\rfloor+\lfloor n_1\rfloor-\lfloor p\rfloor+\delta_{\bar{m}_1}(r)+\delta_{\bar{n}_1}(T-r)+i}}
Y((1+z)^{L(0)}u, z) v\\
=&\sum_{i=0}^{\lfloor p\rfloor}(-1)^{i}
    \binom{\lfloor m_1\rfloor+\lfloor n_1\rfloor-\lfloor p\rfloor-1+\delta_{\bar{m}_1}(r)+\delta_{\bar{n}_1}(T-r)+i}{i} \\
&\cdot \operatorname{Res}_{y} \frac{e^{(\lfloor m_1\rfloor+\delta_{\bar{m}_1}(r)+r / T)y}}{(e^y-1)^{\lfloor m_1\rfloor+\lfloor n_1\rfloor-\lfloor p\rfloor+\delta_{\bar{m}_1}(r)+\delta_{\bar{n}_1}(T-r)+i}}
Y(e^{yL(0)}u, e^y-1) v\\
=&\sum_{i=0}^{\lfloor p\rfloor}(-1)^{i}
    \binom{\lfloor m_1\rfloor+\lfloor n_1\rfloor-\lfloor p\rfloor-1+\delta_{\bar{m}_1}(r)+\delta_{\bar{n}_1}(T-r)+i}{i} \\
&\cdot \operatorname{Res}_{y} \frac{e^{(\lfloor m_1\rfloor+\delta_{\bar{m}_1}(r)+r / T)y}}{(e^y-1)^{\lfloor m_1\rfloor+\lfloor n_1\rfloor-\lfloor p\rfloor+\delta_{\bar{m}_1}(r)+\delta_{\bar{n}_1}(T-r)+i}}
Y[u,y] v
\end{align*}
where we used the change of variables $z = e^y - 1$. A similar computation shows
\[
u \circ_{g,m}^{\,n} v = \Res_y \frac{e^{(\lfloor m \rfloor + \delta_{\bar{m}}(r) + r/T) y}}{(e^y - 1)^{\lfloor m \rfloor + \lfloor n \rfloor + \delta_{\bar{m}}(r) + \delta_{\bar{n}}(T - r) + 1}} Y[u, y] v.
\]
Moreover, by~\cite[Theorem~4.2.1]{Z1}, we have $L[-1] = L(-1) + L(0)$. Comparing the definitions of $A_{g,n}(V, \omega)$, $\tilde{A}_{g,n}(\exp(V, \omega))$, etc., we conclude that the two constructions coincide. This completes the proof.
\end{proof}

\begin{thm} \label{thm9.9}
Let $(V, \omega)$ be a vertex operator algebra with an automorphism $g$ of finite order $T$, and let $n, m \in (1/T)\mathbb{N}$. Then:

{\rm(1)} The associative algebras $\tilde{A}_{g,n}(V, \omega)$ and $A_{g,n}(V, \omega)$ are isomorphic;

{\rm(2)}    The bimodules $\tilde{A}_{g,n,m}(V, \omega)$ and $A_{g,n,m}(V, \omega)$ are isomorphic.
 
\end{thm}

\begin{proof}
Let $B_j$ ($j \in \mathbb{N}$) be the rational numbers defined by
\[
\log(1 + y) = \left( \exp\left( \sum_{j=1}^\infty B_j y^{j+1} \frac{d}{dy} \right) \right) y.
\]
Set $L_+(B) := \sum_{j=1}^\infty B_j L(j)$. By the change-of-variable formula in~\cite{H1},
\[
Y[u, z] = e^{-L_+(B)} Y\bigl(e^{L_+(B)} u, z\bigr) e^{L_+(B)},
\]
which implies that the vertex algebras $(V, Y, \mathbf{1}, \omega)$ and $\exp(V, \omega)$ are isomorphic via the map $e^{L_+(B)}$. Consequently,
\[
\tilde{A}_{g,n}(V, \omega) \cong \tilde{A}_{g,n}\bigl(\exp(V, \omega)\bigr),
\quad
\tilde{A}_{g,n,m}(V, \omega) \cong \tilde{A}_{g,n,m}\bigl(\exp(V, \omega)\bigr).
\]
By Lemma~\ref{lem9.8}, the right-hand sides equal $A_{g,n}(V, \omega)$ and $A_{g,n,m}(V, \omega)$, respectively. Hence the desired isomorphisms follow.
\end{proof}

\begin{cor}
Let $(V, \omega)$ and $(V, \omega')$ be two vertex operator algebra structures on the same underlying vertex algebra $V$, and let $g$ be an automorphism of both structures of order $T$. Then for any $n, m \in (1/T)\mathbb{N}$, the algebras $A_{g,n}(V, \omega)$ and $A_{g,n}(V, \omega')$ are isomorphic, and the bimodules $A_{g,n,m}(V, \omega)$ and $A_{g,n,m}(V, \omega')$ are isomorphic.
\end{cor}

\begin{proof}
Since the two structures share the same vertex algebra $(V, Y, \mathbf{1})$, we have
\[
\tilde{A}_{g,n}(V, \omega) = \tilde{A}_{g,n}(V, \omega'), \quad
\tilde{A}_{g,n,m}(V, \omega) = \tilde{A}_{g,n,m}(V, \omega').
\]
Applying Theorem~\ref{thm9.9} to both conformal vectors $\omega$ and $\omega'$, we deduce that
\[
A_{g,n}(V, \omega) \cong \tilde{A}_{g,n}(V, \omega) = \tilde{A}_{g,n}(V, \omega') \cong A_{g,n}(V, \omega'),
\]
and similarly for the bimodules. This completes the proof.
\end{proof}

It is known from~\cite{Z3,XH1} that for vertex operator algebras,
$
O_{g,n,m}(V) = O'_{g,n,m}(V).
$
Combined with Theorem~\ref{thm9.9}, this yields:

\begin{cor}
Conjecture~\ref{conj7.8} holds for vertex operator algebras.
\end{cor}

\begin{rmk}
For a vertex operator algebra $V$, the proof of equality $O_{g,n,m}(V) = O'_{g,n,m}(V)$ relies crucially on the twisted regular representation theory developed by Li and Sun~\cite{LS1}. Resolving Conjecture~\ref{conj7.8} for general vertex algebras would require an analogous theory. However, the existing framework is deeply tied to the conformal structure—particularly through contragredient modules—and thus does not readily generalize to arbitrary vertex algebras.
\end{rmk}

Combining Theorems~\ref{thm6.4},~\ref{thm9.5} and~~\ref{thm9.9}, we obtain:

\begin{thm}
There exists a bijection between irreducible admissible $g$-twisted $V$-modules and irreducible $(1/T)\mathbb{N}$-graded $g$-twisted $\phi$-coordinated $V$-modules.
\end{thm}


\end{document}